\documentclass[11pt,final]{amsart}
\usepackage[includeheadfoot, top = 1in, bottom = 1in, left =1in, right = 1in]{geometry}
\usepackage{amsfonts,amscd,amssymb,amsmath,amsthm, mathrsfs}
\usepackage{graphicx,tikz,bbm,bm}
\usepackage{xparse}
\usepackage{enumerate,subfigure}
\usepackage{multirow,stmaryrd}
\usepackage{url}
\newcounter{alg}

\renewcommand{\vec}{\bm }

\newtheorem{theorem}{Theorem}[section]
\newtheorem{definition}{Definition}
\newtheorem{remark}{Remark}

\newtheorem{lemma}[theorem]{Lemma}
\newtheorem{proposition}[theorem]{Proposition}
\newtheorem{corollary}[theorem]{Corollary}

\newcommand{\mfd}{\mathfrak d}
\newcommand{\R}{\mathbb{R}}
\newcommand{\C}{\mathbb{C}}

\newcommand{\N}{\mathbb{N}}

\newcommand{\LP}{\operatorname{L}}

\newcommand{\Id}{\operatorname{Id}}

\newcommand{\sd}{\mathrm{d}}

\newcommand{\Exp}{\mathbb{E}}

\newcommand{\prob}{\mathbb{P}}

\newcommand{\Prob}{\prob}
\renewcommand{\Pr}{\prob}
\DeclareDocumentCommand \one { o }
{%
\ensuremath
\IfNoValueTF {#1}
{\mathbf{1}  }
{\ensuremath{\mathbf{1}\left\{ {#1} \right\} }}%
}

\newcommand{\lawequals}{\overset{\mathscr{L}}{=}}
\DeclareDocumentCommand{\Prto} {o} {
\IfNoValueTF {#1}
 {\overset{\Pr}{\longrightarrow}}
 { \xrightarrow[ #1 \to \infty]{\Pr }}
}
\DeclareDocumentCommand{\Asto} {o} {
\IfNoValueTF {#1}
 {\overset{\operatorname{a.s.}}{\longrightarrow}}
 {
 \xrightarrow[ #1 \to \infty]{\operatorname{a.s.} }
 }
}
\DeclareDocumentCommand{\Mgfto} {o} {
\IfNoValueTF {#1}
{\overset{\operatorname{mgf}}{\longrightarrow}}
{ \xrightarrow[ #1 \to \infty]{\operatorname{mgf} }}
}

\DeclareDocumentCommand{\Wkto} {o} {
\IfNoValueTF {#1}
 {\overset{(d)}{\longrightarrow}}
 { \xrightarrow[ #1 \to \infty]{(d) }}
}
\DeclareDocumentCommand{\To} {o} {
\IfNoValueTF {#1}
 {\rightarrow}
 { \xrightarrow[]{#1 \to \infty }}
}

\DeclareDocumentCommand \LPto { O{1} }
{\overset{\operatorname{\LP^{#1}}}{\longrightarrow}}

\title[Universality for the conjugate gradient and MINRES algorithms]{Universality for the conjugate gradient and MINRES algorithms on sample covariance matrices}
 
\author{Elliot Paquette}
\address{The Ohio State University, Columbus, OH} 
\email{paquette.30@osu.edu}
  
\author{Thomas Trogdon}
\address{University of Washington, Seattle, WA}
\email{trogdon@uw.edu}

\thanks{This work was supported in part by NSF DMS-1753185, DMS-1945652 (TT).}



\usepackage{hyperref}
\usepackage{showkeys}
\begin{document}
\maketitle
\begin{abstract}
    We present a probabilistic analysis of two Krylov subspace methods for solving linear systems.  We prove a central limit theorem for  norms of the residual vectors that are produced by the conjugate gradient and MINRES algorithms when applied to a wide class of sample covariance matrices satisfying some standard moment conditions.  The proof involves establishing a four moment theorem for the so-called spectral measure, implying, in particular, universality for the matrix produced by the Lanczos iteration.  The central limit theorem then implies an almost-deterministic iteration count for the iterative methods in question.
\end{abstract}

\section{Introduction}
Sample covariance matrices are one of the oldest class of random matrices.  One can trace their theory at least back to the seminal work of Wishart \cite{WISHART1928}.  Specifically, Wishart considered matrices of the form
\begin{align}\label{eq:wishart}
    W = \frac{1}{M} X X^T
\end{align}
where $X$ is an $N \times M$ matrix whose entries are independent and identically distributed (iid) standard normal random variables.  Such matrices provide an estimator for the covariance matrix of the columns of $X$, and the Wishart distribution can play the role of the null distribution in covariance estimation.  Wishart matrices arise in other settings too, and particularly relevant to this paper, they appear in the seminal work
of Goldstine and von Neumman \cite{Goldstine1951} on the numerical inversion of matrices.

Recently, there has been increasing interest in understanding how algorithms from numerical linear algebra and beyond act on random matrices.  Specifically,  this allows one to give a precise average-case analysis of the algorithms, replacing the standard worst-case estimates/bounds.  For non-iterative methods such as Gaussian elimination, one looks for average-case bounds on rounding errors (see \cite{Sankar2006}, for example).  For iterative methods, more questions can be asked, the most basic of which is the question, ``In exact arithmetic, how many iterations are required, on average, to solve a problem?"  The simplex method from linear programming was addressed in this context by many authors \cite{Borgwardt1987,Smale1983,Spielman2001}.  In these works, the notion of average-case is typically restricted to one ensemble, or distribution.  Indeed, the natural criticism of a simple average-case analysis is that the outcome could be ensemble-dependent, and thus it only has predictive power for a small subset of real-world phenomena.

So, in the context of average-case analysis, it becomes important to show that any arbitrary modeling choices made in defining the ensemble have a limited effect.
In the probability literature, this concept is called \emph{universality}, and it has been studied extensively for many years.  The most famous example of universality is the central limit theorem which states that for sufficiently large $M$, the sums
\begin{align*}
    S_M = \frac{1}{M} \sum_{j=1}^M X_j
\end{align*}
for iid $(X_j)_{j \geq 1}$ concentrate on the mean of $X_1$ (and hence $X_j$ for every $j$) and have small fluctuations of size $M^{-1/2}$ about this mean that are asymptotically normally distributed.  This is true, as soon as the random variables have a finite second moment, and more to the point, it does not depend on any further information about the distribution beyond its first two moments.  It can be argued that this particular universality explains the peculiar prevalence and usefulness of the normal distribution in statistics and nature.

Universality has been featured as a particularly important central feature of random matrix theory, especially in the last 20 years.  Many quantities, such as the largest eigenvalue of $W$, are universal --- they have fluctuations that are independent of the distribution on entries of $W$, with some mild moment conditions. The specific statement for the largest eigenvalue $\lambda_1(W)$ of $W$ is
\begin{align} \label{eq:largest_eval}
    \lim_{M \to \infty} \mathbb P \left( c_d N^{2/3}\left(\lambda_1(W) - (1 + \sqrt{\mfd })^2 \right) \leq t \right) = F_1(t), \quad \mfd = \frac{N}{M},
\end{align}
where $F_1(t)$ is the cumulative distribution function for the Tracy--Widom $(\beta = 1)$ distribution (see \cite{Bai2010}, for example).  Here we suppose that $\mfd \To[M] d$ where $0 < d < \infty$.  If we chose $X$ to have complex entries ($W = \frac {1}{M} XX^*$) then we would arrive at the Tracy--Widom $(\beta = 2)$ distribution. Specifying real versus complex through $\beta =1$ versus $\beta = 2$ is common practice in the random matrix literature and we continue this practice in the current work.

Universality was first combined with the average-case analysis of algorithms in \cite{DiagonalRMT}, then expanded in \cite{Deift2014a}, with rigorous results presented in \cite{Deift2017b,Deift2018}.  See \cite{Deift2017c} for a review.  Here we summarize a result found in \cite{Deift2017b} concerning the power method.  The power method itself is the simple iteration
\begin{align*}
    \vec y_k &= W \vec x_{k-1}, \quad k = 1,2,\ldots\\
    \nu_k &=  \vec y_k^T  \vec x_{k-1},\\
    \vec x_k &= \vec y_k/\|\vec y_k\|_2,
\end{align*}
where $\vec x_0$ is a starting unit vector that is often, in practice, chosen randomly.  If, for example, $W$ is positive definite, then $\nu_k \to \lambda_1(W)$ as $k \to \infty$.  A relevant question is to understand how many iterations are required to properly approximate $\lambda_1(W)$.  Given the halting time
\begin{align*}
    T(W,\vec x_0, \epsilon) = \min \{k:  |\nu_k - \nu_{k-1}| < \epsilon^2 \},
\end{align*}
 a result from \cite{Deift2017b} gives the distributional limit
 \begin{align}\label{eq:Tlim}
    \lim_{N \to \infty} \mathbb P \left( \frac{T(W,\vec x_0,\epsilon)}{\tilde c_d N^{2/3}(\log \epsilon - \frac 2 3 \log N)}\leq t \right) = F_\beta^{\mathrm{gap}}(t), \quad \epsilon \leq N^{-5/3 - \sigma},
\end{align}
for $t \geq 0$, $\sigma > 0$ and a constant $\tilde c_d$.  Here $F_\beta^{\mathrm{gap}}(t)$ can be expressed in terms of the limiting distribution of $\frac{1}{N^{2/3}(\lambda_1(W) - \lambda_2(W))}$. But, more importantly, $F_\beta^{\mathrm{gap}}(t)$ only depends on $\beta$ and not on the precise distribution on the entries of $X$.  One may also consider the distribution of  $\nu_k - \nu_{k-1}$  as $M \to \infty$ and ask whether it is universal.

The purpose of this article is three-fold.
\begin{itemize}
\item  We present a full derivation of distributional formulae for the conjugate gradient algorithm (CGA) and the MINRES algorithm applied to linear systems $W \vec x = \vec b$ where $W$ is distributed as in \eqref{eq:wishart}, addressing both the real and complex cases.  A formula for the CGA applied to the normal equations $W \vec x = \frac{X}{\sqrt{M}} \vec b$ is also given.  This elementary derivation pulls on many well-known results at the intersection of numerical linear algebra and random matrix theory.  In particular, the derivation involves many algorithms that are well-known to the applied mathematics community:  the QR factorization, Golub--Kahan bidiagonalization, singular value decomposition, Lanczos iteration and Cholesky factorization.
\item We then show how universality theorems for the so-called anisotropic local law \cite{Knowles2017} can be upgraded to give universality theorems for the moments of discrete measures that arise in the Lanczos and conjugate gradient algorithms.  This is the key component in showing that the behavior determined in the asymptotic analysis of the formulae in the case of Gaussian matrices indeed persists for a wide class of non-Gaussian matrices giving universality for the norms of residual and error vectors for the CGA and MINRES algorithms.  In the well-conditioned case (i.e., $\mfd \To[M] d \in (0,1)$), the number of iterations of the algorithm to achieve a tolerance $\epsilon$ (i.e., the halting time) is almost deterministic.
\item Because the calculations are so explicit and the estimates are so exact, this work can be viewed as a benchmark for the average-case analysis of an algorithm.  This shows that it is indeed possible to completely analyze an algorithm, in a specific regime, applied to wide class of random matrix distributions.
\end{itemize}

Currently, the small $\epsilon$ (i.e., $\epsilon = \epsilon_M \To[M] 0$) behavior of the CGA and MINRES algorithms on Wishart matrices is open.  By this, we are referring to determining the (asymptotic) distribution on the number of iterations required to achieve a tolerance of $\epsilon$.  Numerical experiments indicate that a universality statement analogous to \eqref{eq:Tlim} holds for the CGA provided $M$ and $N$ are scaled appropriately \cite{Deift2014}, the limiting distribution is conjectured to be Gaussian \cite{Deift2015} and the leading-order behavior is conjectured in \cite{Menon2016}.  

So, in this paper we focus on fixed $\epsilon$ while running the algorithms $O(1)$ steps.  The leading-order analysis along these lines was completed for Gaussian entries in \cite{Deift2019b}.  This confirmed that the deterministic analysis of Beckermann and Kuijlaars \cite{Beckermann2001} (see also \cite{Kuijlaars2006}) holds in the random setting with overwhelming probability. In this paper we improve upon and simplify the results in \cite{Deift2019b} in many respects.  In particular, our exact distributional formulae (see Theorem~\ref{t:GaussianOnly}) can be used to establish many, but not all, of the results in \cite{Deift2019b}.  We then prove that the leading-order results in \cite{Deift2019b} are universal and provide the universal distributional limit (after rescaling) for the fluctuations. This also provides a universal, almost-deterministic halting time (see Remarks~\ref{rem:CG_halting} and \ref{rem:MINRES_halting}).  Such almost-deterministic halting times for the CGA were first observed in \cite{Deift2015} and proved in \cite{Deift2019b} in the Gaussian case.  See \cite{Paquette2020a} for similar results in the case of gradient descent. 

While our analysis for the CGA and MINRES algorithms is focused on sample covariance matrices of the form \eqref{eq:wishart}, many other distributions should be analyzable.  One example would be $I + \gamma G$, $G = \frac{ X + X^T}{\sqrt{2N}}$ where $X$ is an $N \times N$ iid Gaussian matrix.  This is the shifted Gaussian orthogonal ensemble.  For a definite and well-conditioned problem, one should choose $\gamma < 1/2$.  Another interesting case is for sample covariance matrices $T^{1/2}XX^TT^{1/2},$ for deterministic positive definite matrix $T,$ which correspond to sample covariance matrices with non-identity covariance.  But in either of these cases, one can run the Lanczos iteration on it and ask about the distribution on the tridiagonalization that results.  The leading-order behavior is implied by \cite{Vargas2019}.   And indeed, as we discuss, this fact is qualitatively implied by the fact that the entries in the Lanczos matrix are differentiable functions of the moments of an associated spectral measure.  

    The paper is laid out as follows.  In this section we fix notation, introduce the Gaussian distributions from which we perturb and discuss the algorithms that we will analyze.  We present our main results in Theorems~\ref{t:deterministic}, \ref{t:GaussianOnly}, \ref{t:main-lo} and \ref{t:main}.  The section closes with a numerical demonstration of the theorems.  In Section~\ref{sec:SCMs} we introduce the notion of sample covariance matrices and the moment matching condition and discuss properties of basic algorithms applied to Gaussian matrices.  Section~\ref{sec:OPs} gives some properties of orthogonal polynomials that are critical in our calculations.  Section~\ref{sec:cg} gives a deterministic description of the CGA and MINRES algorithm along with the derivation of formulae for the errors that result from the algorithms. The main probabilistic contribution of the paper is in Section~\ref{sec:univ}.  It comes in the form of a ``four moment theorem" for the spectral measure.  Lastly, Section~\ref{sec:analysis} completes the proofs of our main theorems.

\subsection{Notation}

Throughout this article we use boldface, e.g., $\vec y$, to denote vectors.  The norm $\|\vec y\|_2^2 = \vec y^* \vec y$ gives the usual $2$-norm. The expression $W > 0$ indicates that $W$ is a real-symmetric or complex-Hermitian positive definite matrix.  And $W$ then induces an important norm $\|\vec y \|_W^2 = \vec y^* W\vec y$.  We then use $\lambda_1(W) \geq \lambda_2(W) \geq \cdots \lambda_N(W)$ to denote the eigenvalues of $W$.

The notation $\mathcal N_\beta(\mu,\sigma^2)$ refers to a real ($\beta =1)$ or complex ($\beta = 2$) normal random variable with mean $\mu$ and variance $\sigma^2$ and the symbol $\lawequals$ refers to equality in law.  The notation $x_M \Wkto[M] y$ denotes convergence in distribution, or weak convergence.  Additionally, since we will be using $\vec e_k$ to denote error vectors arising in the approximate solution of linear systems, we use $\vec f_1, \ldots, \vec f_n$ to denote the standard basis of $\mathbb R^n$ where $n$ is inferred from context. The notation $\chi_{\beta k}$ is used to denote the chi distribution with $\beta k$ degrees of freedom parameterized\footnote{Parameterizing a distribution is expressing it as a transformation of well-understood random variables.} by
\begin{align*}
    \chi_{\beta k} \lawequals \left(\sum_{j=1}^k |X_j|^2 \right)^{1/2},
\end{align*}
where $(X_j)_{j=1}^k$ are iid $\mathcal N_{\beta}(0,1)$ random variables.

We also encounter settings where the size of a random matrix or vector is increasing as a parameter $M \to \infty$.  We say that, for example, $(x_j)_{j=1}^M = : \vec x_M \Wkto[M] \vec y$, $\vec y = (y_j)_{j=1}^\infty$ in the sense of convergence of finite-dimensional marginals if for any finite set $S$ of integers
\begin{align*}
  (x_j)_{j \in S}  \Wkto[M] (y_j)_{j \in S}.
\end{align*}
This notion is very convenient as it allows one to bypass dimension mismatches between processes. Lastly, we will use subblock notation $X_{i:k,j:\ell}$ to denote the subblock of the matrix $X$ that contains rows $i$ through $k$ and columns $j$ through $\ell$.

\subsection{The Wishart distributions}

Suppose $X$ is an $N \times M$ matrix of iid $\mathcal N_\beta(0,1)$ normal random variables.  Then we say that $X \lawequals \mathcal G_\beta(N,M)$, and we say $W = XX^*/M$ has the $\beta$-Wishart distribution and write $W \lawequals \mathcal W_{\beta}(N,M).$   The $\beta$-Wishart distributions in the cases\footnote{The case $\beta = 4$ can be introduced using quarternions.} $\beta = 1,2$ has many important properties that we will use extensively.  In addition, classical algorithms from numerical linear algebra act on these matrices in a way that allows for explicit (distributional) calculations.

\subsection{The conjugate gradient and MINRES algorithms}  

The CGA \cite{Hestenes1952} is an iterative method to solve a linear system $W \vec x = \vec b$ where $W > 0$.  Supposing exact arithmetic, the algorithm is simplest to characterize in its varational form.  Define the Krylov subspace
\begin{align}
    \mathcal K_k = \mathrm{span} \{ \vec b, W \vec b, \ldots, W^{k-1} \vec b\}.\label{eq:variational}
\end{align}
Then the $k$th iterate, $\vec x_k$, of the CGA satisfies\footnote{Here we are characterizing the CGA with $\vec x_0 = 0$.}
\begin{align*}
    \vec x_k = \mathrm{argmin}_{\vec y \in \mathcal K_k} \| \vec x - \vec y \|_W.
\end{align*}
In Section~\ref{sec:cg} the algorithm that is often used to compute $\vec x_k$ effectively is presented but since our analysis assumes exact arithmetic, this algorithm is not needed to perform the analysis.

The MINRES algorithm (see Algorithm~\ref{a:minres} below) is another iterative method that works with $\mathcal K_k$ by again producing a sequence 
sequence of vectors
\begin{align*}
    \vec x_1 \to \cdots \to \vec x_k,
\end{align*}
but for the MINRES algorithm each vector $\vec x_k$ solves,
\begin{align*}
    \vec x_k &= \mathrm{argmin}_{\vec y \in \mathcal K_k} \|\vec b - W \vec y \|_2.
\end{align*}

For both the CGA and the MINRES algorithm we use the notation $\vec r_k (W,\vec b) := \vec b - W \vec x_k$ and $\vec e_k(W,\vec b) := \vec x - \vec x_k$ to denote the residual and error vectors, respectively.

\subsection{Main results}

We first establish some deterministic formulae.  The result for $\|\vec r_k\|_2$ in the CG algorithm is entirely classical as it encapsulates a well-known relation between $b_{k-1}$ in Algorithm~\ref{a:cga} below and the entries in the matrix generated by the Lanczos procedure (see \cite{Meurant2019}, for example).  The proof is found in Sections~\ref{sec:CG-non}, \ref{sec:MINRES-non} and \ref{sec:CG-ne}.  See Algorithm~\ref{a:lanczos} and the surrounding text for a discussion of the Lanczos iteration.
\begin{theorem}[Deterministic formulae]\label{t:deterministic}
Consider the Lanczos iteration applied to the pair $(W,\vec b)$ with $W > 0$ and $\| \vec b \|_2 = 1$. Suppose the iteration terminates at step $n \leq N$ producing a matrix $T = T(W,\vec b)$.   Let $T = HH^T$ be the Cholesky factorization (see Algorithm~\ref{a:chol} below) of $T$ where
\begin{align*}
    H = \begin{bmatrix} \alpha_0 \\
   \beta_0 & \alpha_1 \\
   & \beta_1 & \alpha_2 \\
   && \ddots & \ddots\\
   &&& \beta_{n-2} & \alpha_{n-1}
   \end{bmatrix}.
\end{align*}
\begin{enumerate}[(a)]
    \item For the CGA  on $W \vec x = \vec b$ with $\vec x_0 = 0$, for $k < n$,
    \begin{align*}
        \|\vec r_k(W,\vec b)\|_2 &= \prod_{j=0}^{k-1} \frac{\beta_j}{\alpha_j},\\
        \|\vec e_k(W,\vec b)\|_W &= \|\vec r_k(W,\vec b)\|_2\sqrt{\vec f_1^* (L_{k} L_{k}^T)^{-1} \vec f_1}, \quad L_{k} = H_{k+1:n,k+1:n}.
    \end{align*}
    \item For the MINRES algorithm on $W \vec x = \vec b$, for $k < n$,
    \begin{align*}
      \|\vec r_k(W,\vec b)\|_2 = \left(\displaystyle\sum_{j=0}^k \prod_{\ell = 0}^{j-1} \frac{\alpha_\ell^2}{\beta_\ell^2} \right)^{-1/2}.
    \end{align*}
\end{enumerate}
 And $\vec r_{n} = 0$.
\end{theorem}

\begin{theorem}[CG and MINRES on $\mathcal W_{\beta}(N,M)$]\label{t:GaussianOnly}
Suppose $W \lawequals \mathcal W_{\beta}(N,M)$ with $N \leq M$ and $\vec b \in \mathbb R^N$ ($\beta =1$) or $\vec b \in \mathbb C^{N}$ ($\beta = 2$) non-zero.  Let
$\alpha_j \lawequals \chi_{\beta(M-j)}$, $\beta_j \lawequals \chi_{\beta(N - j -1)}$, $j = 0,1,\ldots$ be independent and $k < N$.
\begin{enumerate}[(a)]
\item For the CGA applied to $W \vec x = \vec b$ with $\vec x_0 = 0$,
\begin{align*}
    \|\vec r_k(W,\vec b)\|_2 &\lawequals \|\vec b\|_2\prod_{j=0}^{k-1} \frac{\beta_j}{\alpha_j},\\
    \|\vec e_k(W,\vec b)\|_W &\lawequals \Sigma_k^{-1} \|\vec r_k\|_2, \quad \Sigma_k^{-1} \lawequals \frac{\sqrt{\beta M}}{\chi_{\beta(M-N + 1)}},
\end{align*}
where $\Sigma_k^{-1}$ is independent of $\alpha_j,\beta_j$, $j = 0,1,\ldots,k-1$ but dependent on $\alpha_j,\beta_j$, $j \geq k$.
\item For the MINRES algorithm applied to\footnote{We use the convention that $\prod_{\ell=0}^{-1} \equiv 1$.} $W \vec x = \vec b$,
\begin{align*}
      \|\vec r_k(W,\vec b)\|_2\lawequals \left(\displaystyle\sum_{j=0}^k \prod_{\ell = 0}^{j-1} \frac{\alpha_\ell^2}{\beta_\ell^2} \right)^{-1/2} \|\vec b\|_2.
    \end{align*}
\item   Now suppose $\vec b \in \mathbb R^M$ ($\beta =1$) or $\vec b \in \mathbb C^{M}$ ($\beta = 2$) is non-zero, and $X \lawequals \mathcal G_{\beta}(N,M)$, $N \leq M$.  For the CGA applied to $W \vec x = \frac{X}{\sqrt{M}}\vec b$, $W = \frac{XX^*}{M}$,
\begin{align*}
    \left\|\vec e_k\left( W, \frac{X}{\sqrt{M}}\vec b \right)\right\|_W &\lawequals \Delta_{N,M} \left(\displaystyle\sum_{j=0}^k \prod_{\ell = 0}^{j-1} \frac{\alpha_\ell^2}{\beta_\ell^2} \right)^{-1/2} \|\vec b\|_2,
\end{align*}
where $\Delta_{N,M} \lawequals \frac{\chi_{\beta N}^2}{\chi_{\beta M}^2}$ may have non-trivial correlations with $\alpha_j,\beta_j$, $j =0,1,2,\ldots$ but does not depend on $k$.
\end{enumerate}
\end{theorem}
\begin{remark}
From Theorem~\ref{t:GaussianOnly}(c) we obtain a complete parameterization of the relative errors
\begin{align*}
    \frac{\left\|\vec e_k\left( W, \frac{X}{\sqrt{M}}\vec b \right)\right\|_W}{\left\|\vec e_0\left( W, \frac{X}{\sqrt{M}}\vec b \right)\right\|_W} \lawequals\left(\displaystyle\sum_{j=0}^k \prod_{\ell = 0}^{j-1} \frac{\alpha_\ell^2}{\beta_\ell^2} \right)^{-1/2}.
\end{align*}
\end{remark}

To state the next couple results, we define the parameter $\mfd = N/M$.

\begin{theorem}[Universality to leading order]\label{t:main-lo}
Let $W = XX^*$ where $X$ is an $N \times M$ random matrix $N \leq M$, $\mfd \To[M] d \in (0,1]$, with independent real ($\beta = 1$) or complex ($\beta = 2$) entries.  Suppose, in addition, that there exists constants $\{C_p\}_{1}^\infty$ so that \emph{all} entries of $X$ satisfy, for non-negative integers $\ell,p$,
\begin{equation}\label{eq:mm_lo}
\begin{aligned}
&  \Exp (\Re X_{ij})^\ell(\Im X_{ij})^p = \Exp (\Re Y)^\ell(\Im Y)^p, \\
& Y \lawequals \mathcal N_\beta(0,1/M), \quad \ell + p \leq 2,\\
& \Exp |\sqrt{M} X_{ij}|^p \leq C_p, \quad \text{ for all } p \in \N.
\end{aligned}
\end{equation}
For any sequence $\vec b = \vec b_N$ of unit vectors, in the sense of convergence of finite-dimensional marginals:
\begin{enumerate}[(a)]
\item For the CGA\footnote{We do not discuss $\|\vec r_0\|_2$ here because $\vec r_0 = \vec b$.}
\begin{align*}
    \left( \|\vec e_k(W,\vec b)\|_W^2 \right)_{k \geq 0} \Wkto[M] \left( \frac{d^k}{1-d} \right)_{k \geq 0}, ~~ d \neq 1, \quad \left( \|\vec r_k(W,\vec b)\|_2^2 \right)_{k \geq 1} \Wkto[M] \left( {d^k} \right)_{k \geq 1}.
\end{align*}
\item For the MINRES algorithm
\begin{align*}
\left( \|\vec r_k(W,\vec b)\|_2^2 \right)_{k \geq 1} \Wkto[M] \left( d^k\frac{1-d}{1-d^{k+1}} \right)_{k \geq 1}.
\end{align*}
\item For the CGA applied to the normal equations
\begin{align*}
 \left( \left\|\vec e_k\left(W,\frac{X}{\sqrt{M}}\vec b\right)  \right\|_W^2 \right)_{k \geq 0} \Wkto[M] \left( d^{k+1} \frac{1-d}{1-d^{k+1}} \right)_{k \geq 0}.
\end{align*}
\end{enumerate}
\end{theorem}

The case $d = 1$ in Theorem~\ref{t:main-lo} is treated by continuity, $d \uparrow 1$.  To state our last limit theorem, we must define the limit processes.  Let $\mathcal G = (Z_k)_{k=1}^\infty$ be a process of independent $\mathcal N_1(0,1)$ random variables.  Define three new processes $\mathcal G^{\vec e}= (Z_k^{\vec e})_{k=0}^\infty$, $\mathcal G^{\vec r,\mathrm{CG}}= (Z_j^{\vec r,\mathrm{CG}})_{j=1}^\infty$ and $\mathcal G^{\vec r,\mathrm{MINRES}}= (Z_j^{\vec r,\mathrm{MINRES}})_{j=1}^\infty$ via
\begin{align*}
    Z_k^{\vec e} &= \frac{d^k}{1-d} \left[\sum_{j=k}^\infty d^{j-k} (Z_{2j}/\sqrt{d} - Z_{2j+1}) + \sum_{j=1}^{k-1} (Z_{2j}/\sqrt{d} - Z_{2j-1}) - Z_{2k-1} \right],  \\
     Z_k^{\vec r,\mathrm{CG}} &=  {d^k} \left[\sum_{j=0}^{k-1} \left(Z_{2j+2}/\sqrt{d} -  Z_{2j+1} \right)\right],\quad k > 0, \quad Z_0^{\vec r,\mathrm{CG}} = 0,\\
     Z_k^{\vec r,\mathrm{MINRES}} &=  \left(\frac{1-d}{1 - d^{k+1}} \right)^2 \sum_{j=0}^k d^{2(k-j)} Z_j^{\vec r,\mathrm{CG}}.
\end{align*}

\begin{theorem}[Universality of the fluctuations]\label{t:main}
Let $W = XX^*$ where $X$ is an $N \times M$ random matrix, $N \leq M$, $\mfd \To[M] d \in (0,1]$ with independent real ($\beta = 1$) or complex ($\beta = 2$) entries.  Suppose, in addition, that there exists constants $\{C_p\}_{1}^\infty$ so that \emph{all} entries of $X$ satisfy, for non-negative integers $\ell,p$,
\begin{equation}\label{eq:thm_mm}
\begin{aligned}
&  \Exp (\Re X_{ij})^\ell(\Im X_{ij})^p = \Exp (\Re Y)^\ell(\Im Y)^p, \\
& Y \lawequals \mathcal N_\beta(0,1/M), \quad \ell + p \leq 4,\\
& \Exp |\sqrt{M} X_{ij}|^p \leq C_p, \quad \text{ for all } p \in \N.
\end{aligned}
\end{equation}
For any sequence $\vec b = \vec b_N$ of unit vectors, in the sense of convergence of finite-dimensional marginals:
\begin{enumerate}[(a)]
\item For the CGA
\begin{align*}
    \sqrt{\frac{\beta M}{2}} \left(\|\vec e_k(W,\vec b)\|_W^2 -\frac{\mfd^k}{1-\mfd}\right)_{k \geq 0} \Wkto[M] \mathcal G^{\vec e}, \quad d \neq 1,\\
    \sqrt{\frac{\beta M}{2}} \left(\|\vec r_k(W,\vec b)\|_2^2 -\mfd^k\right)_{k \geq 1} \Wkto[M] \mathcal G^{\vec r,\mathrm{CG}}.
\end{align*}
\item For the MINRES algorithm (the case $\mathfrak d = 1$ obtained using continuity)
\begin{align*}
    \sqrt{\frac{\beta M}{2}} \left(\|\vec r_k(W,\vec b)\|_2^2 -\mfd^k \frac{1-\mfd}{1-\mfd^{k+1}}\right)_{k \geq 1} \Wkto[M] \mathcal G^{\vec r,\mathrm{MINRES}}.
\end{align*}
\item For the CGA applied to the normal equations (the case $\mathfrak d = 1$ obtained using continuity)
\begin{align*}
    \sqrt{\frac{\beta M}{2}} \left( \frac{\left\|\vec e_k\left(W,\frac{X}{\sqrt{M}}\vec b\right)\right\|_W^2}{\left\|\vec e_0\left(W,\frac{X}{\sqrt{M}}\vec b\right)\right\|_W^2} -  \mfd^k \frac{1-\mfd}{1-\mfd^{k+1}}\right)_{k \geq 1} \Wkto[M] \mathcal G^{\vec r,\mathrm{MINRES}}.
\end{align*}
\end{enumerate}
\end{theorem}

The proofs of the previous theorems can be roughly summarized as follows.  Modulo some technical issues in dealing with correlations, Theorem~\ref{t:GaussianOnly} can be directly used, with the asymptotics of independent chi random variables, to prove Theorem~\ref{t:main-lo} and \ref{t:main} in the case $\sqrt{M} X \lawequals \mathcal G_{\beta}(N,M)$.  Asymptotic correlations are addressed in Proposition~\ref{prop:dist-limit}.  Associated to $(W, \vec b)$, $W > 0, \|\vec b\|_2 = 1$ is a weighted empirical spectral measure (see \eqref{eq:weighted_mu} below).  The orthogonal polynomials with respect to this measure satisfy a three-term recurrence which when assembled into a Jacobi matrix coincides with the output $T_n(W, \vec b)$ of the Lanczos iteration (see Proposition~\ref{prop:lanczos_equiv} below).  Then the well-known fact that the entries in the three-term recurrence Jacobi matrix can be recovered as algebraic functions of the moments of the measure is used (see \eqref{eq:sandl}).  This means that the entries in the Cholesky factorization of $T_n(W \vec ,b)$ are (generically) differentiable functions of the moments of the weighted empirical spectral measure. Then Theorem~\ref{t:GC} establishes universality for the moments and hence for the entries in the Cholesky factorization.  More specifically, this implies that Proposition~\ref{prop:dist-limit} holds in the non-Gaussian case, implying our theorems.

Some important remarks are in order. 
\begin{remark}\label{rem:CG_halting}
Let $W$, $d<1$ and $\vec b$ be as in Theorem~\ref{t:main-lo}.  Define two CGA halting times
  \begin{align*}
      t^{\vec e} (W,\vec b,\epsilon) = \min \{k: \|\vec e_k(W,\vec b)\|_W < \epsilon\},\quad t^{\vec r} (W,\vec b,\epsilon) = \min \{k: \|\vec r_k(W,\vec b)\|_2 < \epsilon\}.
  \end{align*}
  If $\epsilon^2 \neq d^k/(1-d)$ for all $k$
  \begin{align*}
        \lim_{N \to \infty} \mathbb P \left( t^{\vec e} (W,\vec b,\epsilon) = \left\lceil \frac{\log \epsilon^2 (1-d)}{\log d} \right\rceil \right) = 1,
  \end{align*}
  and if $\epsilon^2 = d^k/(1-d)$ for some $k$ then
  \begin{align*}
        \lim_{N \to \infty} \mathbb P \left( t^{\vec e} (W,\vec b,\epsilon) = \left\lceil \frac{\log \epsilon^2 (1-d)}{\log d} \right\rceil \right) &= \frac 1 2,\\
        \lim_{N \to \infty} \mathbb P \left( t^{\vec e} (W,\vec b,\epsilon) = 1 + \left\lceil \frac{\log\epsilon^2 (1-d)}{\log d} \right\rceil\right) & = \frac 1 2.
  \end{align*}
  Similarly, if $\epsilon^2 \neq d^k$ for all $k$
  \begin{align*}
        \lim_{N \to \infty} \mathbb P \left( t^{\vec r} (W,\vec b,\epsilon) = \left\lceil \frac{2 \log \epsilon)}{\log d} \right\rceil \right) = 1,
  \end{align*}
  and if $\epsilon = d^k$ for some $k$ then
  \begin{align*}
        \lim_{N \to \infty} \mathbb P \left( t^{\vec r} (W,\vec b,\epsilon) = \left\lceil \frac{2 \log \epsilon}{\log d} \right\rceil \right) = \frac 1 2 = \lim_{N \to \infty} \mathbb P \left( t^{\vec r} (W,\vec b,\epsilon) = 1 + \left\lceil \frac{2 \log \epsilon}{\log d} \right\rceil\right).
  \end{align*}
  \end{remark}
  
  \begin{remark}\label{rem:MINRES_halting}
  Let $W$, $d< 1$ and $\vec b$ be as in Theorem~\ref{t:main-lo}.  Define the MINRES halting time
  \begin{align*}
      t^{\mathrm{MINRES}}(W,\vec b, \epsilon) = \min\{ k : \|\vec r_k(W,\vec b)\|_2 < \epsilon \}.
  \end{align*}
  Then if $\epsilon ^2 \neq d^k \frac{1 - d}{1- d^{k+1}}$ for all $k$
  \begin{align*}
        \lim_{N \to \infty} \mathbb P \left( t^{\mathrm{MINRES}} (W,\vec b,\epsilon) = \left\lceil \frac{\log \frac{\epsilon^2}{1-d+\epsilon^2 d}}{\log d} \right\rceil \right) = 1,
  \end{align*}
  and if $\epsilon ^2 = d^k \frac{1 - d}{1- d^{k+1}}$ for some $k$ then
  \begin{align*}
        \lim_{N \to \infty} \mathbb P \left( t^{\mathrm{MINRES}} (W,\vec b,\epsilon) = \left\lceil \frac{\log \frac{\epsilon^2}{1-d+\epsilon^2 d}}{\log d} \right\rceil \right) &= \frac 1 2,\\
        \lim_{N \to \infty} \mathbb P \left( t^{\mathrm{MINRES}} (W,\vec b,\epsilon) = 1 + \left\lceil \frac{\log \frac{\epsilon^2}{1-d+\epsilon^2 d}}{\log d} \right\rceil\right) & = \frac 1 2.
  \end{align*}
  And so, the MINRES algorithm, using the halting criterion $\|\vec r_k\|_2 < \epsilon$ will run for approximately $\frac{\log (1 - d + \epsilon^2 d)}{\log d}$ fewer steps than the CGA.
  \end{remark}

 \begin{remark}
 Let $W$, $d< 1$ and $\vec b$ be as in Theorem~\ref{t:main}. For fixed $k$
  \begin{align}
    \sqrt{\frac{\beta M}{2}} \left( \|\vec e_k(W,\vec b)\|_W^2 - \frac{d^k}{1-d} \right) &\Wkto[M] \mathcal N_1(0,\sigma_{k,\vec e}^2), \quad d \neq 1, \notag\\
    \sigma_{k,\vec e}^2 &= \frac{d^{2k}}{(1-d)^2} \left[\frac{1}{d(1-d)} + (k-1) \left( 1 + \frac{1}{d}\right) + 1\right],\notag\\
     \sqrt{\frac{\beta M}{2}} \left( \|\vec r_k(W,\vec b)\|_2^2 - d^k \right) &\Wkto[M] \mathcal N_1(0,\sigma_{k,\vec r}^2),\notag\\
     \sigma_{k,\vec r}^2 &= k d^{2k} \left(1 + \frac{1}{d}\right).\label{eq:res_var}
  \end{align}
  \end{remark}
  
  \begin{remark}
 The expression for $Z_k^{\vec r, \mathrm{MINRES}}$ can be written as 
 \begin{align*}
     Z_k^{\vec r, \mathrm{MINRES}}  = d^{2k} \left( \frac{1-d}{1-d^{k+1}} \right)^2 \sum_{\ell = 0}^k \frac{d^{-k} - d^{-\ell}}{1-d} \left( Z_{2 \ell + 2}/\sqrt{d} - Z_{2\ell +1} \right),\quad k >0.
 \end{align*}
Let $W$, $d< 1$ and $\vec b$ be as in Theorem~\ref{t:main}. For fixed $k$ it then follows that 
  \begin{align}
     \sqrt{\frac{\beta M}{2}} &\left( \|\vec r_k(W,\vec b)\|_2^2 - d^k \right) \Wkto[M] \mathcal N_1(0,\hat \sigma_{k,\vec r}^2),\notag\\
     \hat \sigma_{k,\vec r}^2 &= \frac{(1-d) d^{2 k-1} \left(2 d^{k+1}+2 d^{k+2} - d^{2 k+2}-d^2 (k+1)-2
   d+k\right)}{\left(1-d^{k+1}\right)^4}.\label{eq:min_res_var}
  \end{align}
  \end{remark}

\begin{remark}
Additionally, one obtains the formulae for the CGA applied to $W \vec x= \vec b$, $\| \vec b \|_2 = 1$, $W \lawequals \mathcal W_{\beta}(N,M), N \leq M$
\begin{align*}
    \mathbb E \|\vec r_k(W, \vec b)\|_2 &= \prod_{j = 0}^{k-1}  \frac{\Gamma\left( \frac{\beta (N -j -1) + 1}{2} \right)}{\Gamma \left( \frac{\beta (N -j -1)}{2} \right)} \frac{ \Gamma\left( \frac{ \beta(M-j)-1}{2} \right)}{ \Gamma \left( \frac{\beta(M-j)}{2} \right)},\\
    \mathbb E \|\vec e_k(W, \vec b)\|_W &= \sqrt{\frac{\beta M}{2}} \frac{ \Gamma\left( \frac{ \beta(M-N+1)-1}{2} \right)}{ \Gamma \left( \frac{\beta(M-N+1)}{2} \right)} \prod_{j = 0}^{k-1}  \frac{\Gamma\left( \frac{\beta (N -j -1) + 1}{2} \right)}{\Gamma \left( \frac{\beta (N -j -1)}{2} \right)} \frac{ \Gamma\left( \frac{ \beta(M-j)-1}{2} \right)}{ \Gamma \left( \frac{\beta(M-j)}{2} \right)},
\end{align*}
where $\Gamma(z)$ is the Gamma function \cite{DLMF}.  For even moderately large $M$, one needs to use the Beta function to compute these ratios and avoid underflow/overflow.
\end{remark}

\begin{remark}
    For $\mfd \to 1$, the CGA applied to $W \vec x = \vec b$ gives
    \begin{align*}
        \frac{\|\vec r_k(W,\vec b)\|_2}{\|\vec r_0(W,\vec b)\|_2} \Wkto[M] 1.
    \end{align*}
    Thus number of iterations required to hit a tolerance $\epsilon$ increase without bound as $M \to \infty$.  On the other hand, for the MINRES algorithm,
    \begin{align*}
        \frac{\|\vec r_k(W,\vec b)\|_2}{\|\vec r_0(W,\vec b)\|_2} \Wkto[M] \frac{1}{\sqrt{k+1}}.
    \end{align*}
    And so, one expects $k \approx \epsilon^{-2}-1$ iterations to achieve $\|\vec r_k(W,\vec b)\|_2 < \epsilon$.  The same statement holds for the CGA applied to the normal equations when $\mfd \to 1$, when one considers the ratio
    \begin{align*}
        \frac{\left\|\vec e_k\left( W, \frac{X}{\sqrt{M}}\vec b \right)\right\|_W}{\left\|\vec e_0\left( W, \frac{X}{\sqrt{M}}\vec b \right)\right\|_W}.
    \end{align*}
\end{remark}

\begin{remark}
If $\vec b = \vec c/\|\vec c\|_2$ where $\vec c$ has iid, mean-zero entries with a finite (non-zero) variance then one expects \eqref{eq:mm_lo} to be sufficient for Theorem~\ref{t:main} to hold --- the moment matching to order two is sufficient if the right-hand side vector is ``sufficiently" random.
\end{remark}

\subsection{A numerical demonstration}

We demonstrate the essential aspects of Theorem~\ref{t:main}(a) for $\|\vec r_k\|_2$ in Figures~\ref{fig:cg_res_wishart} and \ref{fig:cg_res_mm}.  In these figures we compare the CGA applied to $W \vec x = \vec f_1$ with $W \lawequals \mathcal W_{\beta}(N,M)$ and $W = XX^*/M$ where $X$ has iid entries with $\prob ( X_{ij} = 0 ) = 2/3, \prob (X_{ij} = \pm \sqrt{3} ) = 1/6$.  This discrete distribution, which we refer to as the moment matching distribution, is chosen so that the first four moments of $X_{ij}$ coincide with that of $\mathcal N_1(0,1)$.  The figures demonstrate that $\|\vec r_k\|_2$ concentrates heavily as $M$ increases.

The essential aspects of Theorem~\ref{t:main}(b) are shown in Figures~\ref{fig:cg_res_wishart} and \ref{fig:cg_res_mm}. These figures again give the behavior of the MINRES algorithm and CGA applied to the $\beta = 1$ Wishart distribution and the moment matching distribution.  

Lastly, in Figure~\ref{fig:moment}, for the CGA, we compare the statistics of
\begin{align}\label{eq:rescale}
    \sqrt{M} \left( \frac{\|\vec r_k(W, \vec f_1)\|_2^2}{ \langle \|\vec r_k(W, \vec f_1)\|_2^2 \rangle } - 1 \right),
\end{align}
where $\langle Z \rangle$ represents the sample average of $Z$ over 50,000 samples. Note that if \eqref{eq:thm_mm} holds then 
\begin{align*}
    \sqrt{M} \left( \frac{\|\vec r_k(W, \vec f_1)\|_2^2}{ \langle \|\vec r_k(W, \vec f_1)\|_2^2 \rangle } - 1 \right) \approx \mathcal N_1(0,\sigma_{k,d}/2),
\end{align*}
and we therefore compare the density for $\mathcal N_1(0,\sigma_{k,d}/2)$ with \eqref{eq:rescale} in Figure~\ref{fig:moment}.  In this figure we also include computations with the Bernoulli ensemble: $W = XX^*/M$, $X_{ij}$ iid, $\mathbb P(X_{ij} = \pm 1) = 1/2$ which fails to satisfy \eqref{eq:thm_mm}.

In Table~\ref{tab:moment} we display sample variance of \eqref{eq:rescale} for the three different distributions: Wishart, moment matching and Bernoulli.  In the case of the Wishart and moment matching distributions, the variance is close to the large $M$ limit.  In the case of Bernoulli, the variance is quite different.  This indicates that the moment matching condition is a necessary condition for the limiting the variance to be given by \eqref{eq:res_var}.

\begin{figure}[tbp]
\centering
\subfigure[]{\includegraphics[width=.65\linewidth]{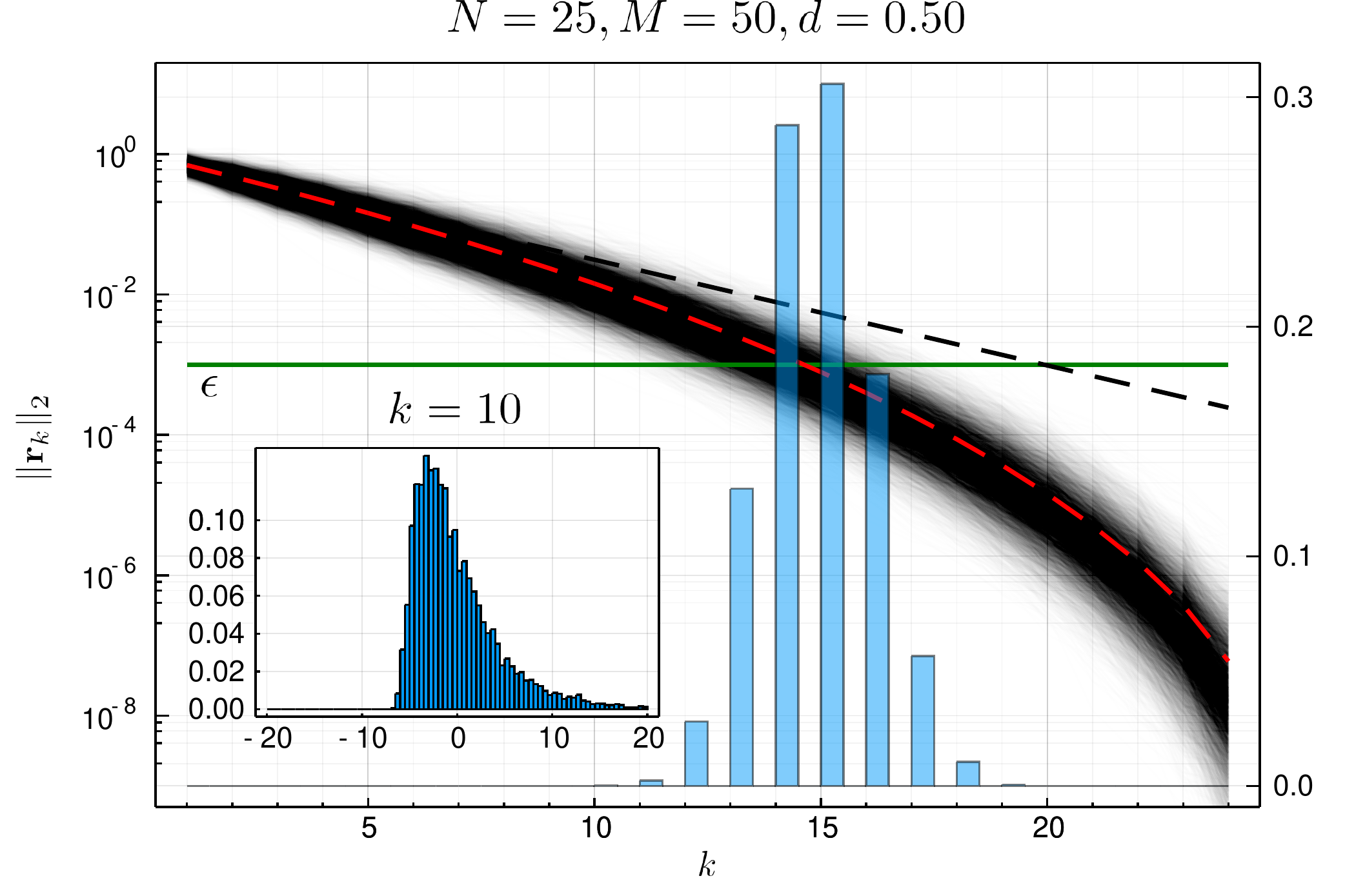}}
\subfigure[]{\includegraphics[width=.65\linewidth]{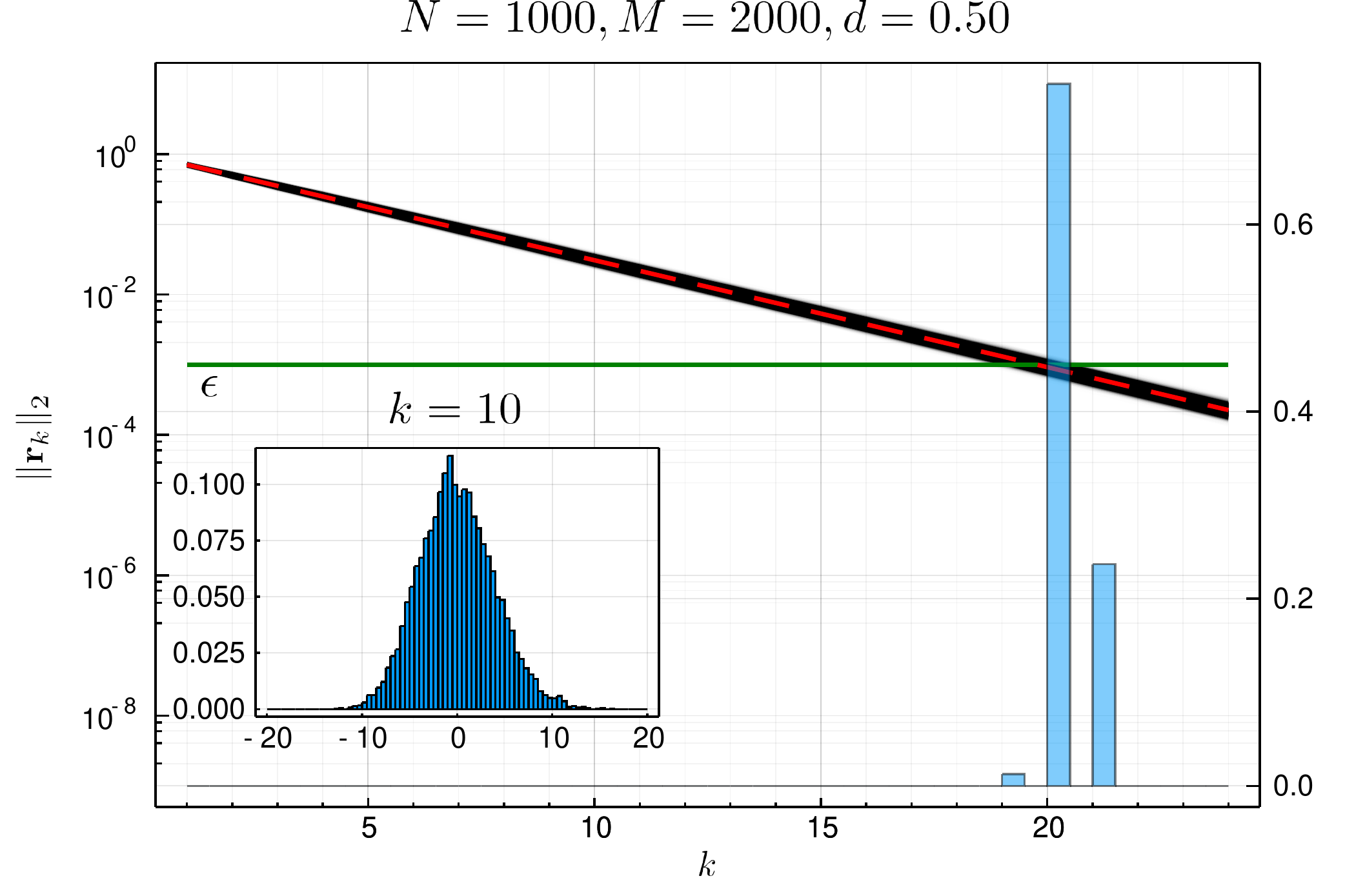}}
    \caption{\label{fig:cg_res_wishart} The CGA applied to $W \vec x = \vec f_1$ were $W \lawequals \mathcal W_{\beta}(N,M)$, $N/M \To[M] d$.  The dashed black curve indicates the large $M$ limit for the error $\|\vec r_k\|_2$ at step $k$ and the dashed red curve gives $\mathbb E \|\vec r_k\|_2$ at step $k$. The shaded gray area is an ensemble of 10000 runs of the method, displaying the norms that resulted.  The overlaid histogram shows the rescaled fluctuations in the error at $k = 10$.  As $M \to \infty$ this approaches a Gaussian density.  Lastly, the histogram in the main frame gives the halting distribution for $\epsilon = 0.001$ (green line).  It is highly concentrated when $N = 1000, M = 2000$.  With these parameters, Remark~\ref{rem:CG_halting} implies that for $M$ large, the algorithm will run for approximately $\left\lceil 2 \frac{\log \epsilon}{\log d} \right\rceil = 20$ iterations. }
\end{figure}
\begin{figure}[tbp]
\centering
\subfigure[]{\includegraphics[width=0.65\linewidth]{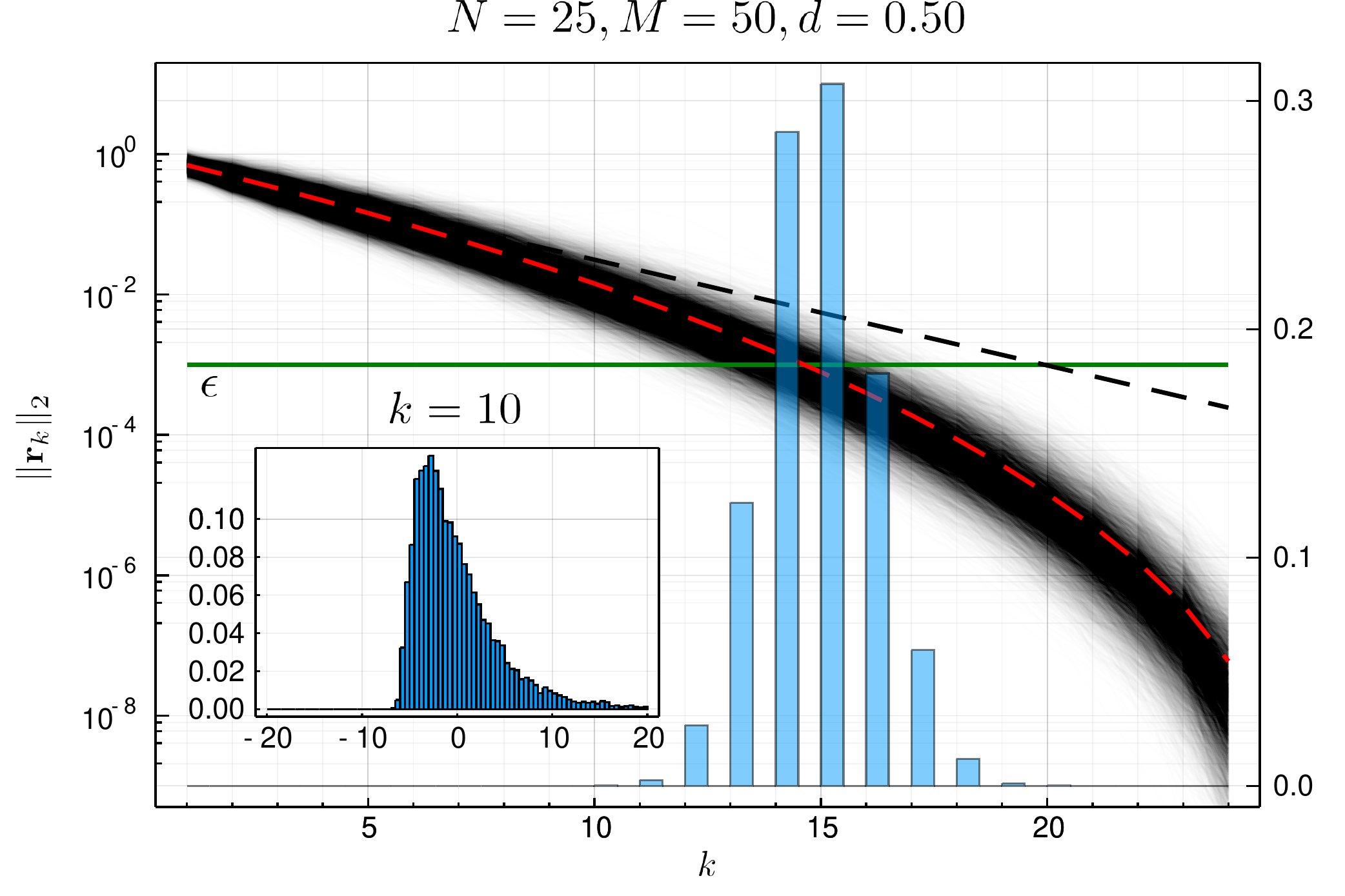}}
\subfigure[]{\includegraphics[width=0.65\linewidth]{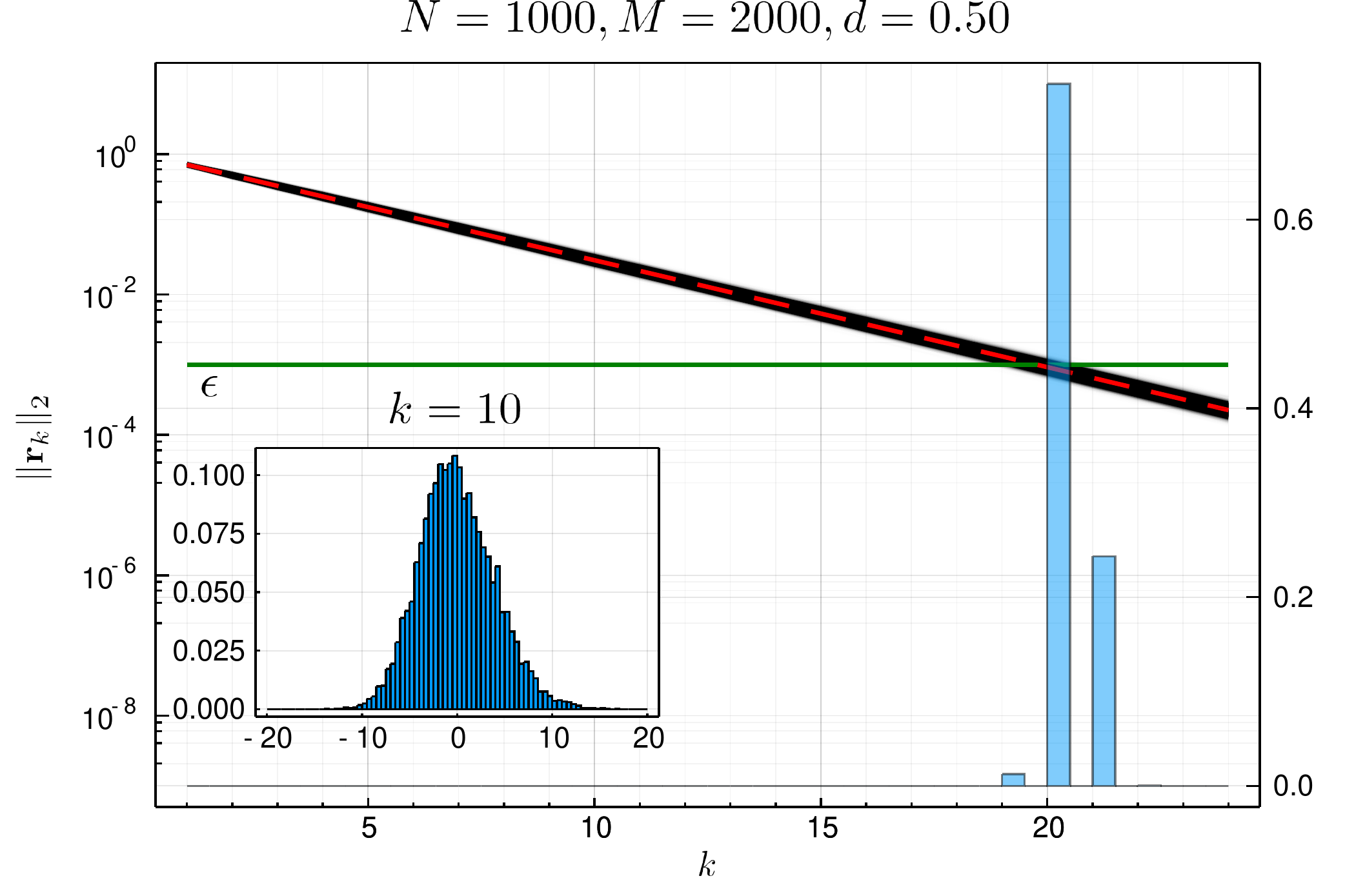}}
    \caption{\label{fig:cg_res_mm} The CGA applied to $W \vec x = M^{-1} XX^* \vec x = f_1$ were $X$ has iid entries with $\prob ( X_{ij} = 0 ) = 2/3, \prob (X_{ij} = \pm \sqrt{3} ) = 1/6$.   The black dashed curve indicates the large $M$ limit for the error $\|\vec r_k\|_2$ at step $k$ the dashed red curve gives $\mathbb E \|\vec r_k\|_2$ at step $k$ in the case of $W \lawequals \mathcal W_{\beta}(N,M)$, for comparison.  The shaded gray area is an ensemble of 10000 runs of the method, displaying the errors that resulted.  The overlaid histogram shows the rescaled fluctuations in the error at $k = 10$.  As $M \to \infty$ this approaches a Gaussian density.  Lastly, the histogram in the main frame gives the halting distribution for $\epsilon = 0.001$ (green line). With these parameters, Remark~\ref{rem:CG_halting} implies that for $M$ large, the algorithm will run for approximately  $\left\lceil 2 \frac{\log \epsilon}{\log d} \right\rceil = 20$ iterations.  }
\end{figure}
\begin{figure}[tbp]
\centering
\subfigure[]{\includegraphics[width=0.65\linewidth]{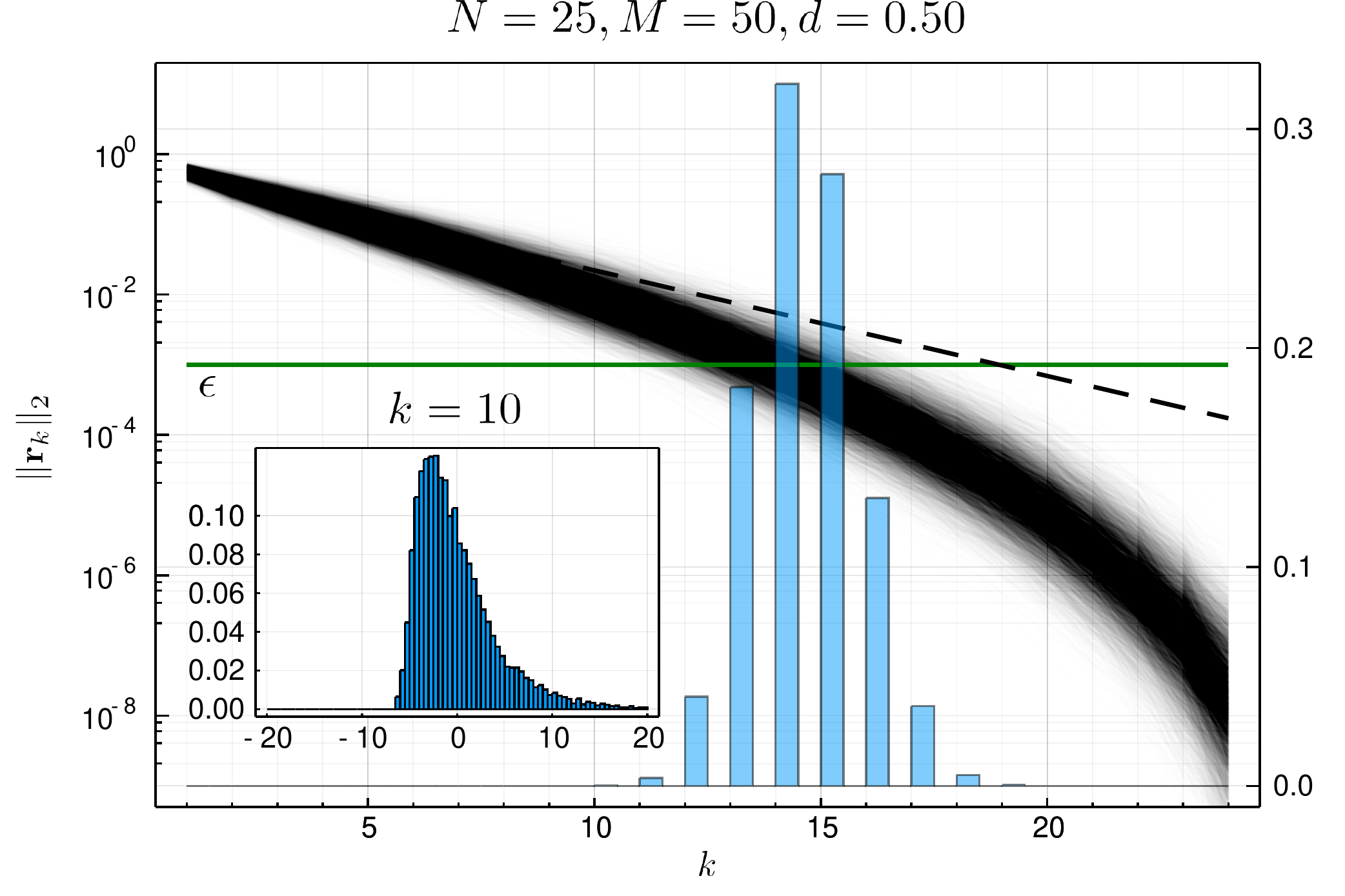}}
\subfigure[]{\includegraphics[width=0.65\linewidth]{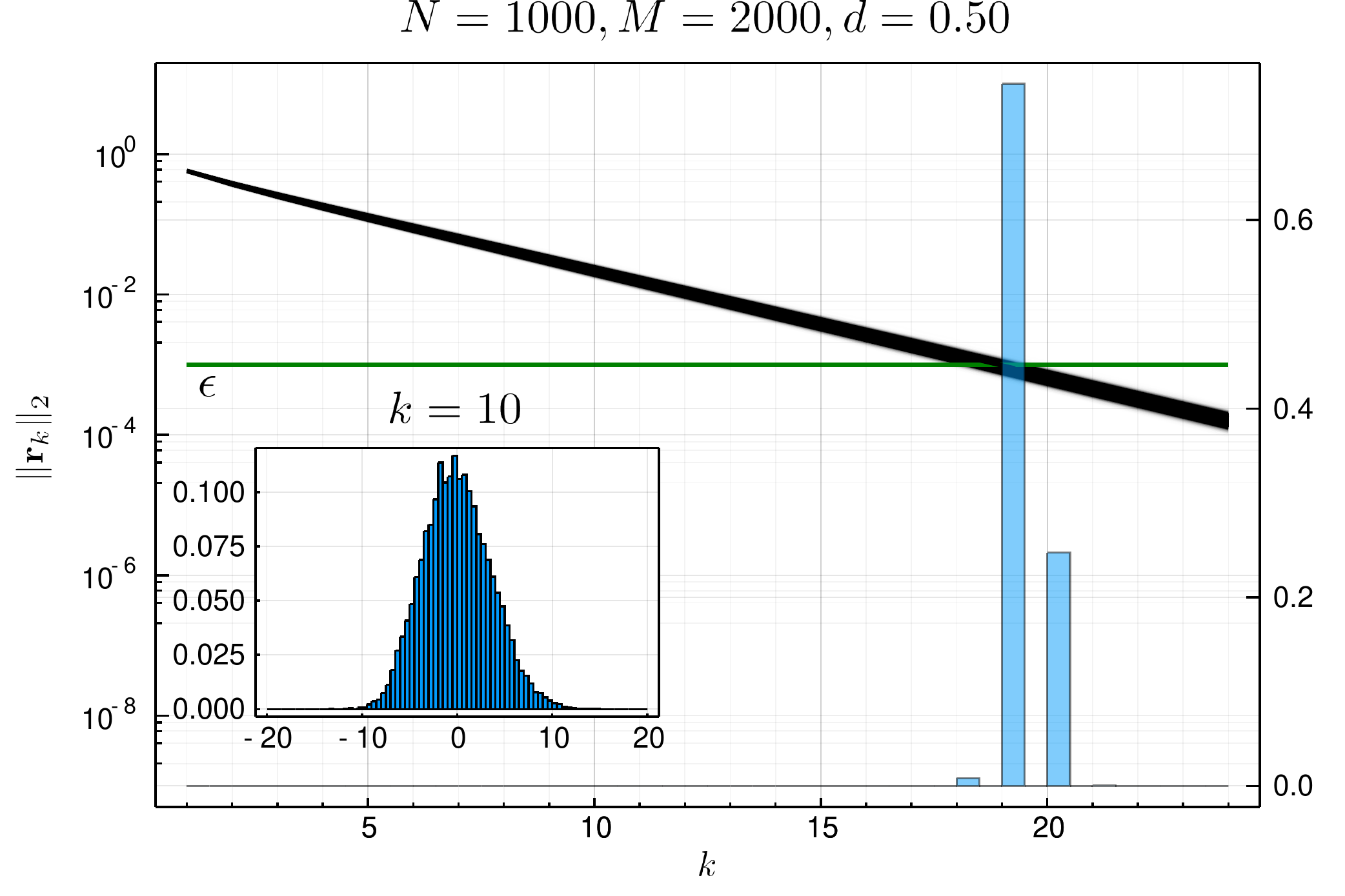}}
    \caption{\label{fig:minres_wishart} The MINRES algorithm applied to $W \vec x = \vec f_1$ were $W \lawequals \mathcal W_{\beta}(N,M).$  The dashed curve indicates the large $M$ limit for the error $\|\vec r_k\|_2$ at step $k$.  The shaded gray area is an ensemble of 10000 runs of the method, displaying the errors that resulted.  The overlaid histogram shows the rescaled fluctuations in the error at $k = 10$.  As $M \to \infty$ this approaches a Gaussian density.  Lastly, the histogram in the main frame gives the halting distribution for $\epsilon = 0.001$ (green line). With these parameters, Remark~\ref{rem:MINRES_halting} implies that for $M$ large, the algorithm will run for approximately $\left\lceil \frac{\log\frac{ \epsilon^2}{1-d+\epsilon^2 d}}{\log d} \right \rceil = 19$ iterations.}
\end{figure}
\begin{figure}[tbp]
\centering
\subfigure[]{\includegraphics[width=0.65\linewidth]{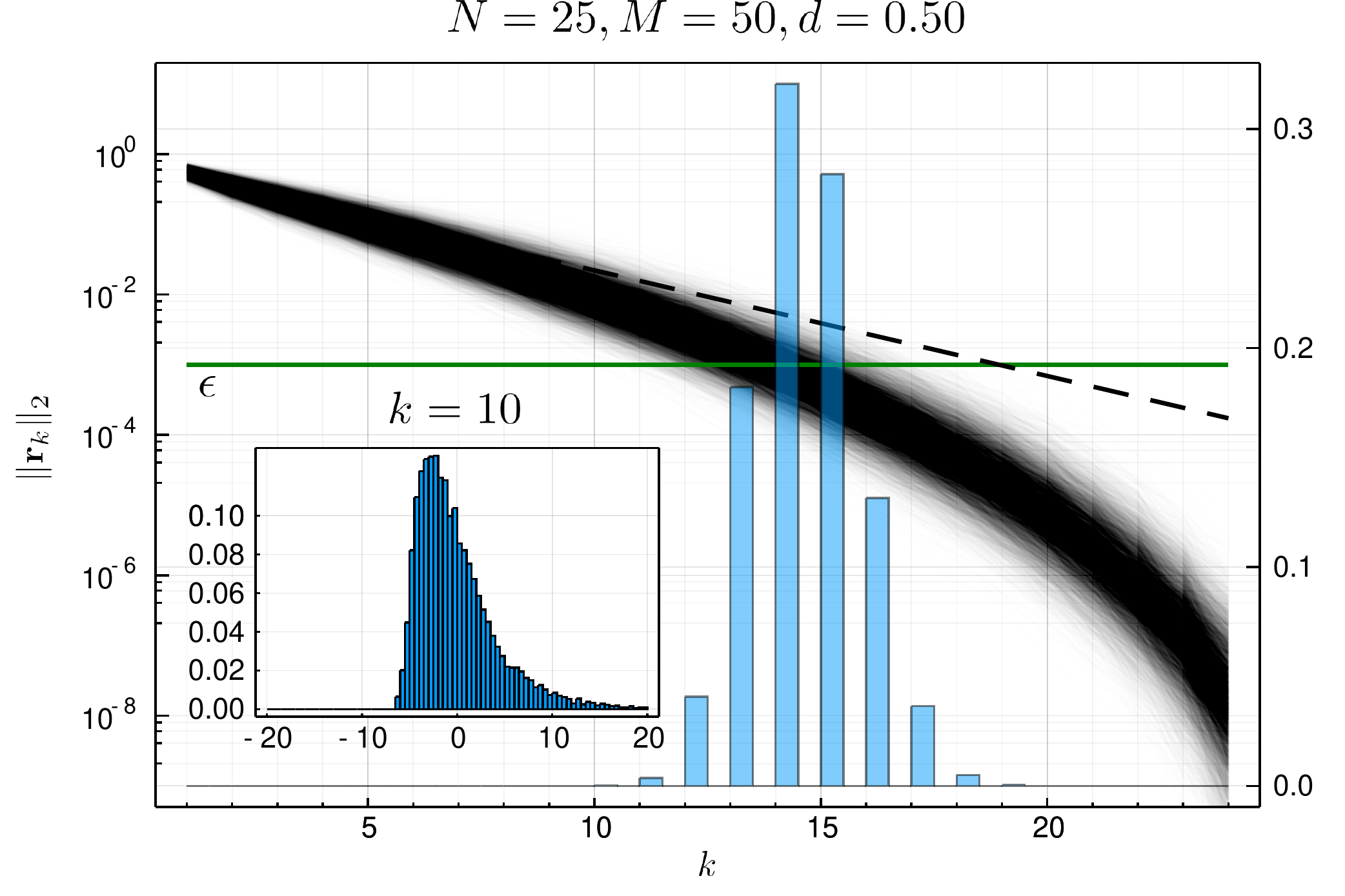}}
\subfigure[]{\includegraphics[width=0.65\linewidth]{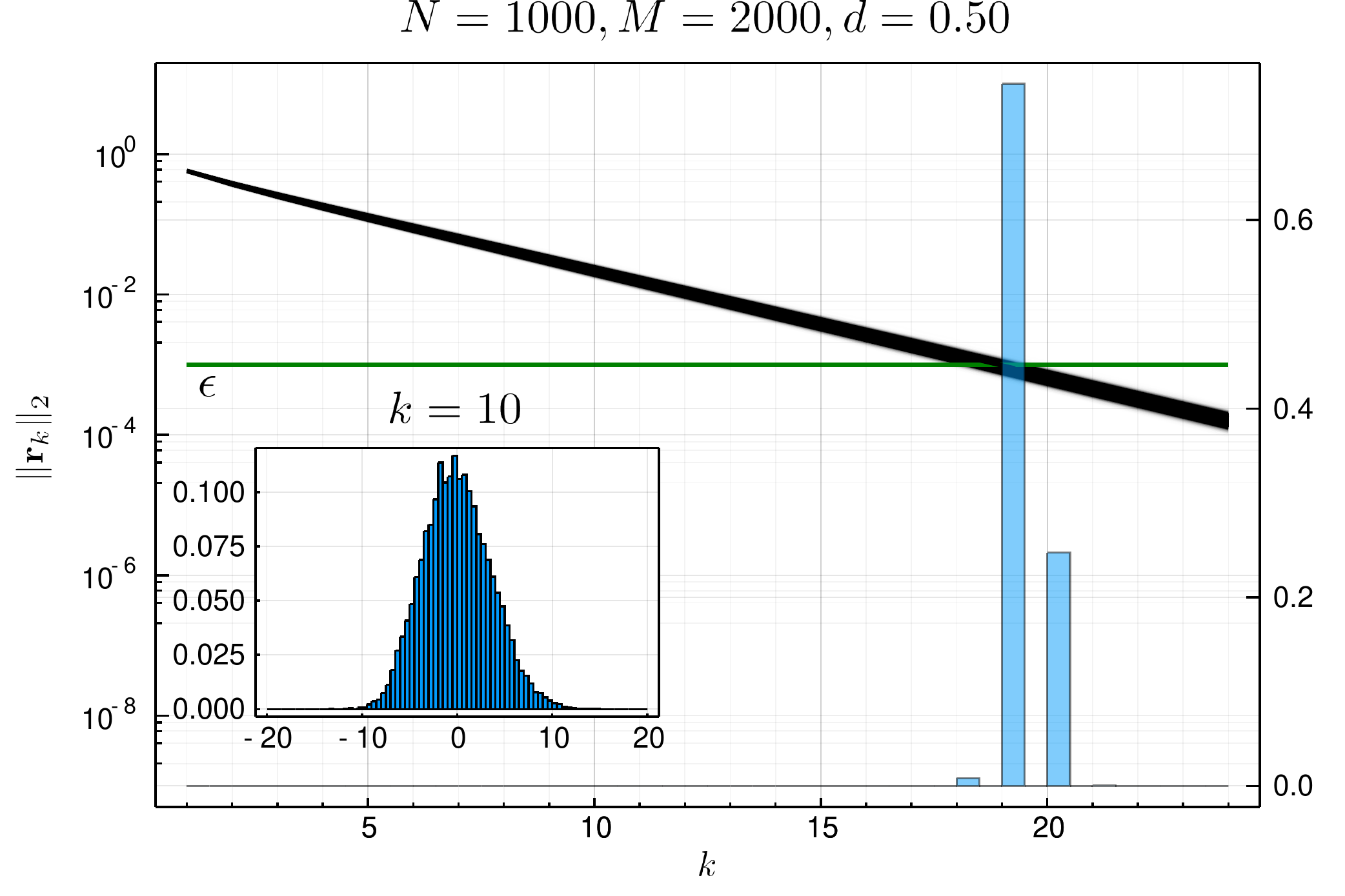}}
    \caption{\label{fig:minres_mm} The MINRES algorithm applied to $W \vec x = M^{-1} XX^* \vec x = f_1$ were $X$ has iid entries with $\prob ( X_{ij} = 0 ) = 2/3, \prob (X_{ij} = \pm \sqrt{3} ) = 1/6$.  The shaded gray area is an ensemble of 10000 runs of the method, displaying the errors that resulted.  The overlaid histogram shows the rescaled fluctuations in the error at $k = 10$.  As $M \to \infty$ this approaches a Gaussian density.  Lastly, the histogram in the main frame gives the halting distribution for $\epsilon = 0.001$ (green line). With these parameters, Remark~\ref{rem:MINRES_halting} implies that for $M$ large, the algorithm will run for approximately $\left\lceil \frac{\log\frac{ \epsilon^2}{1-d+\epsilon^2 d}}{\log d} \right\rceil = 19$ iterations.}
\end{figure}
\begin{figure}[tbp]
\centering
\includegraphics[width=\linewidth]{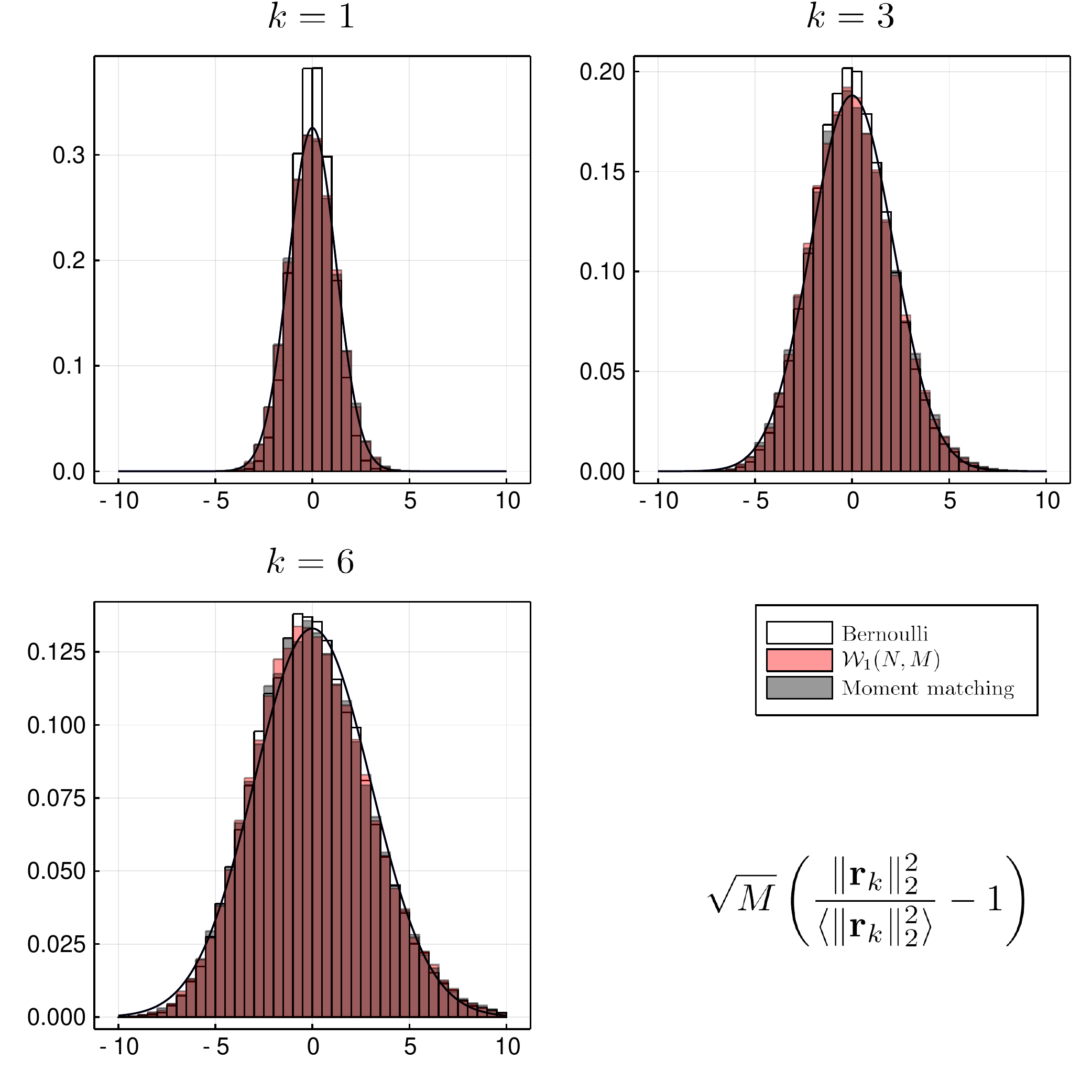}
    \caption{\label{fig:moment} A comparison of the rescaled statistics \eqref{eq:rescale} across three distributions.  Since the Bernoulli ensemble fails to match the moments in \eqref{eq:thm_mm}, we see that it does not match the variance \eqref{eq:res_var}.}
\end{figure}
\begin{table}[tbp]
\centering
\begin{tabular}{|c|c|c|c|c|}
\hline
$k$ & $ k/2 (1 + 1/d)$ & Wishart & Moment matching & Bernoulli\\
\hline
1 & 1.5 & 1.493 & 1.48 &  1.003 \\
\hline
 2& 3.0 & 3.002 & 2.997 & 2.511\\
  \hline
3 &   4.5 & 4.532 & 4.519 & 4.036\\
  \hline
4&  6.0 & 6.040 & 6.039 & 5.527 \\
 \hline
5 &  7.5 & 7.576 & 7.54  & 7.004 \\
 \hline
6&  9.0 & 9.135 & 9.054 & 8.547 \\
 \hline
 \end{tabular}
 \vspace{.05in}
 \caption{\label{tab:moment}  A numerical demonstration of the necessity of the moment matching condition \eqref{eq:thm_mm}.  This table gives the sample variance of \eqref{eq:rescale} across three different distributions for $N = 500, d = 1/2$ and 50,000 samples.}
 \end{table}

\section{Sample covariance matrices and classical numerical linear algebra} \label{sec:SCMs}


A fundamental property of a matrix $X \lawequals \mathcal G_\beta(N,M)$ is its orthogonal ($\beta = 1$) or unitary ($\beta =2$) invariance.  That is, let $Q$ be an $N \times N$ fixed orthogonal matrix then
\begin{align*}
    QWQ^* \lawequals W, \quad W = X X^*.
\end{align*}
If $\beta = 2$, then $Q$ can be a complex unitary matrix. Furthermore, this is true even if $Q$ is random, provided it is independent of $X$.

Let $W \lawequals \mathcal W_\beta(N,M)$ and perform an eigenvalue decomposition $W = U \Lambda U^*$, $U^*U  = I$.  It follows directly from the invariance of the Wishart distribution that the vector
\begin{align*}
    \vec{\omega} = \begin{bmatrix}|U_{11}|^2 \\ \vdots \\ |U_{1n}|^2 \end{bmatrix}, \quad\text{ where }\quad   \left[ U_{ij} \right]_{1 \leq i, j \leq n}= U
\end{align*}
can be parameterized by
\begin{align}\label{eq:paramit}
   \vec {\omega}
   \lawequals
   \frac{\vec {\nu}}{\|\vec {\nu}\|}_1,
\end{align}
where $\vec {\nu}$ is a vector of iid $\chi_\beta^2$ random variables.  This fact is discussed in detail in \cite[Appendix A]{Deift2019b}.

\subsubsection{The eigenvalues of the Wishart distributions}

The global asymptotic eigenvalue distribution of the Wishart distributions is the same, regardless of the choice of $\beta = 1,2$.  The classical setup is the following.  For $W \lawequals \mathcal W_{\beta}(N,M)$, define the (random) empirical spectral measure
\begin{align*}
  \mu_{\mathrm{em}}(\sd\lambda; W) = \frac{1}{N} \sum_{j=1}^N \delta_{\lambda_j(W)}( \sd \lambda).
\end{align*}
Recall the parameter $\mfd = N/M$. 

\begin{definition}
Define the \emph{Marchenko--Pastur} law \emph{for all} $d > 0$ by
\begin{equation}
  \varrho_{d}(\sd x) = \frac{1}{2\pi d}
  \sqrt{
    \frac{ [(x-\gamma_{-})(\gamma_+ - x)]_+}{x^2}}\sd x +  \left[ 1 - \frac{1}{d} \right]_+ \delta_0(\sd x),
  \quad
  \text{where}
  \quad
  \gamma_{\pm} = (1 \pm \sqrt{d})^2
  \label{eq:MP}
\end{equation}
are the spectral edges.  The notation $[\cdot]_+$ refers to the positive part of $(\cdot).$
\end{definition}

The following gives the global eigenvalue distribution (see \cite{Bai2010}, for example):
\begin{theorem} \label{t:classical_MP}
  Suppose that $\mfd \To[M] d \in (0,1]$.  Then
  \begin{align*}
    \mu_{\mathrm{em}}(\sd\lambda;\mathcal W_{\beta}(N,M)) \Wkto[N] \varrho_{d}(\sd \lambda),
  \end{align*}
  almost surely.
\end{theorem}

Historically, the behavior of individual eigenvalues, and gaps between eigenvalues, have been studied extensively.  In the analysis we present it is not necessary to use such detailed microscopic results.  Instead, we need finer results about global properties of the matrix.  One such example is the so-called central limit theorem for linear statistics.  

The Bai-Silverstein \cite{BaiSilverstein2004} central limit theorem for linear statistics of sample covariance matrices shows that for sufficiently smooth functions $f,$
\[
\sum_{j=1}^N f(\lambda_j)- N\int f(x) \varrho_{\mfd }(\sd x)
=
N\int f(x)( \mu_{\mathrm{em}}(\sd x) - \varrho_\mfd(\sd x))
\Wkto[N] \mathcal{N}_1(\mu_f,\sigma_f^2).
\]
The standard deviation $\sigma_f$ can be understood as a weighted Sobolev-1/2 norm of $f$, restricted to the support of the Marchenko-Pastur law.  Other related central limit theorems for linear spectral statistics of sample covariance matrices include \cite{DumitriuEdelman2006,Shcherbina2011,Johansson1998}.  

But the classical central limit theorem for linear statistics involves the empirical spectral measure $ \mu_{\mathrm{em}}(\sd\lambda; W)$ which rarely arises in a numerical or computational context.  What is much more likely to arise is the weighted empirical spectral measure:  for $\vec b \in \mathbb C^N$, $\|\vec b\|_2 = 1$ and $W = W^* \in \mathbb C^{N \times N}$ the weighted empirical spectral measure is given by
\begin{align}\label{eq:weighted_mu}
  \mu = \mu_{\vec b}= \sum_{j=1}^N \omega_j \delta_{\lambda_j}, \quad (\omega_j)_{j =1}^N = | U^* \vec b |^2, \quad W = U \Lambda U^*, \quad U^* U = I, \quad \Lambda = \mathrm{diag}(\lambda_1, \ldots, \lambda_N).
\end{align}
We refer to this as the spectral measure associated to the pair $(W,\vec b)$.

We show in Section 5 that for polynomials $p$ and a sample covariance matrix $W$ with identity covariance and for which $\mfd \to d,$
\[
\sqrt{M}\int f(x) ( \mu_{\vec b}(\sd x) - \varrho_{\mfd }(dx) )
\Wkto[N] \mathcal{N}_1(0,\widehat{\sigma_f}^2).
\]
Note that the rate of the central limit theorem changes dramatically from the case of the central limit theorem for linear statistics.
Although we will not need it, the variance $\widehat{\sigma_f}^2$ can be expressed as $c_{\beta,d} \int f^2(x)\varrho_d(d x)$.  Similar theorems have been proven before, most notably by \cite{ORourkeRenfrewSoshnikov2014} who prove a more general statement in the case that $\vec b$ is a coordinate vector.  There is also \cite{ORourkeRenfrewSoshnikov2013} in which the analogous statement is made for Wigner matrices.  We also mention \cite{Duy} and \cite{DuyShirai} which prove related theorems for Gaussian cases.

While it is natural to assume these statements extend to other classes of test functions beyond polynomials, we will not need them (except for the specific case of $f(x)=1/x,$ which we handle by other means -- note that the extension to analytic functions in a neighborhood of the Marchenko-Pastur law does not need new ideas beyond what is necessary for the polynomial case)

\subsection{Sample covariance matrices with independence}\label{sec:gen_SCMs}

In the current work, we use a restricted definition of a sample covariance matrix.

\begin{definition}\label{def:SCMs}
A real ($\beta = 1$) or complex ($\beta = 2$) sample covariance matrix is given by $W \lawequals XX^*$ where $X$ is an $N \times M$ random matrix with independent entries satisfying
\begin{align*}
    \Exp X_{ij} &= 0, ~~~\Exp (\Re X_{ij}) (\Im X_{ij} ) = 0, ~~~ \Exp (\Re X_{ij})^2 = \frac{1}{\beta M},\\ \Exp |X_{ij}|^2 &= \frac{1}{M},
     \quad \text{and} \quad 
    \Exp |\sqrt{M} X_{ij} |^p \leq C_p, \quad \text{for all } p \in \mathbb N.
\end{align*}
\end{definition}

In some cases, we will need restrictions on the first four generalized moments.

\begin{definition}\label{def:moment}
A sample covariance matrix satisfies the $\beta = 1,2$ moment matching condition if
\begin{align*}
    \Exp (\Re X_{ij})^\ell (\Im X_{ij} )^p = \Exp (\Re Y)^\ell (\Im Y )^p
\end{align*}
where $Y \lawequals \mathcal N_\beta(0,1/M)$, for all choices of non-negative integers $\ell,p$ such that $\ell + p \leq 4$.
\end{definition}

\begin{remark}
To see the necessity of the moment matching condition consider a sample covariance matrix $W' = XX/M$ where $X'$ is $N \times M,$ with $X'_{ij} = \pm 1$ with equal probability and $W \lawequals \mathcal W_{1}(N,M)$.  Then consider the first moments of the spectral measures  $\mu$ and $\mu'$ associated to $(W, \vec f_1)$ and $(W' ,\vec f_1)$, respectively:
\begin{align*}
    \int \lambda \mu(\sd \lambda) &= \frac{1}{M} \vec f_1^TXX^T \vec f_1 \lawequals \frac{\chi_{M}^2}{M},\\
    \int \lambda \mu'(\sd \lambda) &= \frac{1}{M} \vec f_1^TX'X'^T \vec f_1 =1.
\end{align*}
\end{remark}

\subsection{The Golub--Kahan bidiagonalization algorithm} \label{sec:bidiag}

\begin{definition}
A Jacobi matrix is given by
\begin{align*}
    T = \begin{bmatrix} a_0 & b_0 \\
    b_0 & a_1 & b_1\\
    & b_1 & a_2 & \ddots\\
    && \ddots & \ddots \end{bmatrix}.
\end{align*}
It may be finite or semi-infinite.  The entries are real and $b_j > 0$ for $j \geq 0$. 
\end{definition}

A reduction of $W = XX^*$ to a Jacobi matrix can be obtained via the Golub--Kahan bidiagonalization procedure.  The distributional action of this algorithm on the Wishart ensembles $\mathcal W_{\beta}(N,M)$ is given in \cite{Dumitriu2002}.  Specifically, if $W = M^{-1} XX^*\lawequals \mathcal W_{\beta}(N,M)$, $X \lawequals \mathcal G_{\beta}(N,M)$ then there exists unitary matrices $U_1$, $U_2$ such that
\begin{align}\label{eq:GK}
\begin{split}
    U_1 X U_2 &\lawequals \sqrt{\beta} X_{\mathrm{GK}},\\
    \sqrt{\beta} X_{\mathrm{GK}} &\lawequals \left[\begin{array}{ccccc|cccccc} \chi_{\beta M} &&&&& \\
    \chi_{\beta(N-1)} & \chi_{\beta(M-1)} &&&& \\
    & \chi_{\beta(N-2)} & \chi_{\beta(M-2)}  && & \quad & 0 & \quad \\
    && \ddots & \ddots && \\
    &&& \chi_{\beta} & \chi_{\beta(M - N + 1)} \end{array}\right],
\end{split}
\end{align}
where all entries are independent.  Therefore the law of the entries of the tridiagonal matrix $U_1 W U_1^* = M^{-1} U_1 XU_2U_2^*X^* U_1^* = \beta X_{\mathrm{GK}} X_{\mathrm{GK}}^T$ is completely parameterized.


\subsection{The Lanczos iteration}

The Lanczos iteration is another algorithm for obtaining a tridiagonal reduction of a matrix.\\

\centerline{\noindent\fbox{%
\refstepcounter{alg}
    \parbox{.9\textwidth}{%
\flushright \boxed{\text{Algorithm~\arabic{alg}: Lanczos Iteration\label{a:lanczos}}}
\begin{enumerate}
    \item $\vec q_1$ is the initial vector.  Suppose $\|\vec q_1\|_2^2 = \vec q_1^* \vec q_1 = 1$, $W^* = W$.
    \item Set $b_{-1} = 1$, $\vec q_{0} = 0.$
    \item For $k = 1,2,\ldots,n$
    \begin{enumerate}
        \item Compute $\displaystyle a_{k-1} = (W \vec q_k - b_{k-2} \vec q_{k-1})^* \vec q_k$.
        \item Set $\vec v_k = W \vec q_k - a_{k-1} \vec q_k - b_{k-2} \vec q_{k-1}$.
        \item Compute $b_{k-1} = \|\vec v_k\|_2$ and if $b_{k-1} \neq 0$, set $\vec q_{k+1} = \vec v_k/b_{k-1}$, otherwise terminate.
    \end{enumerate}
\end{enumerate}
    }%
}}

\vspace{.1in}
The Lanczos algorithm at step $k \leq N$ produces a matrix $T_k$ and orthogonal vectors $\vec q_1,\ldots,\vec q_k$
\begin{align*}
    Q_k = \begin{bmatrix} \vec q_1 & \vec q_2 & \cdots &\vec q_k \end{bmatrix}, \quad T_k = T_k(W,\vec y_1) = \begin{bmatrix} a_0 & b_0 \\
    b_0 & a_1 & \ddots \\
    & \ddots & \ddots & b_{k-2} \\
    & & b_{k-2} & a_{k-1} \end{bmatrix},
\end{align*}
such that
\begin{align}\label{eq:Tk}
 W Q_k = Q_k T_k + b_{k-1} \vec q_{k+1} \vec f_k^*
\end{align}
We use the notation $T = T(W,\vec q_1) = T_n(W,\vec q_1)$ for the matrix produced when the Lanczos iteration is run for its maximum of $n$ steps.

The following is entirely classical \cite{TrefethenBau}.

\begin{lemma}
Suppose $W$ is a symmetric matrix.  And suppose that the Lanczos iteration does not terminate before step $n \leq N$.  For $k = 1,2,\ldots,n$,
\begin{align*}
    \vec q_1, \ldots, \vec q_k
\end{align*}
is an orthonormal basis for the Krylov subspace $\mathcal K_k = \mathrm{span}\{\vec q_1,W\vec q_1,\ldots,W^{k-1}\vec q_1\}.$
\end{lemma}

The following result gives us the distribution of $T_k$ throughout the Lanczos iteration applied to a Wishart matrix and it is a direct consequence of the invariance of the Wishart distributions.

\begin{theorem}\label{t:lanczos}
Suppose $W  \lawequals \mathcal W_{\beta}(N,M)$.  For any given $\vec q_1 \in \mathbb R^n$ with $\|\vec q_1\|_2=1$(or $\mathbb C^n$ for $\beta = 2$) with probability one, the Lanczos iteration does not terminate if $k < n:= \min\{N,M\}$.  And the distribution on $a_k,b_k$, $k = 0,2,\ldots,n-1$ does not depend on $\vec q_1$.  In a distributional sense it suffices to take $\vec q_1 = \vec f_1$ and therefore the distribution is determined by the Householder tridiagonalization of $W$, i.e., the Golub--Kahan bidiagonalization of $X$.
\end{theorem}

Every $N \times N$ symmetric tridiagonal matrix $T$ produces a probability measure
\begin{align*}
    \mu_T = \sum_{j=1}^N \omega_j \delta_{\lambda_j} 
\end{align*}
where $\lambda_j$'s are the eigenvalues of $T$ and $\omega_j$ is the squared modulus of the first component of the normalized eigenvector associated to $\lambda_j$.  The spectral measure $\mu_T$, $T = T(W,\vec b)$ coincides with the spectral measure associated to the pair $(W,\vec b)$ whenever $\vec b$ is a unit vector.  There is a bijection between such measures and Jacobi matrices \cite{DeiftOrthogonalPolynomials}.

\section{Theory of orthogonal polynomials}\label{sec:OPs}

Let $\mu$ be a Borel probability measure on $\mathbb R$ with finite moments.  The orthonormal polynomials $(p_n)_{n \geq 0}$, $p_n(\lambda) = p_n(\lambda;\mu)$ are constructed by applying the Gram--Schmidt process to the sequence of functions 
\begin{align*}
    \{\lambda \mapsto 1,\lambda \mapsto\lambda,\lambda \mapsto\lambda^2,\ldots\}.
\end{align*}
If the support of $\mu$ contains at least $N$ points then one is guaranteed to be able to construct $(p_0,p_1,\ldots,p_{N-1})$.

\subsection{Hankel determinants, moments and the three-term recurrence}

We now recall the classical fact that the coefficients in a three-term recurrence relation can be recovered as an algebraic function of the moments of the associated spectral measure.  For a given sequence of orthonormal polynomials, $(p_j(\lambda))_{j \geq 0} = (p_j(x;\mu))_{j \geq 0}$ with respect to a measure\footnote{For our purposes it suffices to assume that $\mu$ has compact support.} $\mu$, we have the associated three-term recurrence
\begin{align}\label{eq:3term}
    \lambda p_n(\lambda) = b_n p_{n+1}(\lambda) + a_n p_n(\lambda) + b_{n-1} p_{n-1}(\lambda), \quad n \geq 0, \quad b_n > 0,
\end{align}
with the convention $p_{-1}(\lambda) = 0$ and $b_{-1} = 0$.  Here $b_n = b_n(\mu), a_n = a_n(\mu)$ are called the recurrence coefficients.  We will use the following proposition in a critical way to translate any discussion of the output of the Lanczos iteration to a discussion of orthogonal polynomials.

\begin{proposition}\label{prop:lanczos_equiv}
The three-term recurrence coefficients generated by the spectral measure associated to the pair $(W,\vec b)$, $W > 0, \|\vec b\|_2 = 1$ coincide with the entries of the Lanczos matrix $T(W,\vec b)$.  
\end{proposition}

We write $p_n(\lambda) = \ell_n \lambda^n + s_n \lambda^{j-1} + \cdots$ and find by equating coefficients that
\begin{align*}
    \ell_n = b_n \ell_{n+1},\\
    a_n \ell_n = b_n s_{n+1}.
\end{align*}
Define $D_n$ and $D_n(\lambda)$ by the determinants
\begin{align*}
    D_n &= \det M_n, \quad (M_n)_{ij} = m_{i+j-2}, \quad 1 \leq i,j \leq n+1, \quad m_j(\mu) = m_{j} = \int \lambda^j \mu(\sd \lambda),\\
    D_n(\lambda) &= \det M_n(\lambda),
\end{align*}
and $M_n(\lambda)$ is formed by replacing the last row of $M_n$ with the row vector $[1 ~\lambda ~ \lambda^2 \cdots \lambda^n]$.  Then, it is well-known that \cite{DeiftOrthogonalPolynomials}
\begin{align*}
    p_n(\lambda) = \frac{D_n(\lambda)}{\sqrt{D_n D_{n-1}}},
\end{align*}
and therefore
\begin{align}\label{eq:sandl}
    \ell_n = \sqrt{\frac{D_{n-1}}{D_{n}}}, \quad s_n = \det \tilde M_n,
\end{align}
where $\tilde M_n$ is the matrix formed by removing the last row and second-to-last column of $M_n$.  This shows that $a_n/\sqrt{D_{n-1}}$ and $b^2_n$ are rational functions of determinants of matrices involving only the moments of $\mu$ up to order $2n$.

Associated to the three-term recurrence \eqref{eq:3term} is the Jacobi matrix
\begin{align*}
    T = \begin{bmatrix} a_0 & b_0 \\
    b_0 & a_1 & b_1\\
    & b_1 & a_2 & \ddots\\
    && \ddots & \ddots \end{bmatrix}.
\end{align*}
  Let $T_n$ denote the upper-left $n \times n$ subblock of $T$.  It follows immediately that $T_n$ is a differentiable function of $(m_0,m_1,\ldots,m_{2n})$ on the open subset of $\R^{2n+1}$ where all $D_k > 0$ for $1 \leq k \leq n$.  We also note that
\begin{align}\label{eq:TriMoment}
    \vec f_1^* T^k \vec f_1 = \int \lambda^k \mu(\sd \lambda).
\end{align}
This can be seen by a direct calculation if $T$ is a finite-dimensional matrix.  If $T$ is semi-infinite, then this fact follows from \cite[(2.25)]{DeiftOrthogonalPolynomials}.

\subsection{Monic polynomials and Stieltjes transforms}

The monic orthogonal polynomials associated to a measure $\mu$ are given by
\begin{align}
    \pi_n(\lambda;\mu) = \pi_n(\lambda) &= p_n(\lambda)/\ell_n = \lambda^n + \cdots.
    \end{align}
    We will also need the Stieltjes transform of the monic polynomials
    \begin{align}
     c_n(z;\mu) = c_n(z) &= \int_{\mathbb R} \frac{\pi_n(\lambda)}{\lambda - z} \mu(\sd  \lambda).
\end{align}
With the convention that $b_{0} = 1$, $\pi_{-1} \equiv 0$ and $c_{-1} \equiv -1$ it is elementary that the following recurrences are satisfied for $n = 0,1,2,\ldots,$
\begin{align*}
    &\pi_{n + 1}(\lambda) = (\lambda- a_n) \pi_{n}(\lambda) - b_{n-1}^2 \pi_{n-1}(\lambda), \quad \pi_{0}(\lambda) = 1,\\
    & c_{n + 1}(z) = (z- a_n) c_{n}(z) - b_{n-1}^2 c_{n-1}(z), \quad c_0(z) = \int_{\mathbb R} \frac{\mu(\sd \lambda)}{\lambda - z} .
\end{align*}


\section{The conjugate gradient algorithm and the MINRES algorithm}\label{sec:cg}

In this section we discuss three algorithms: the CGA, the CGA applied to the normal equations and the MINRES algorithm.

\subsection{The CGA}

The actual CGA is given by the following.\\

\centerline{\noindent\fbox{%
\refstepcounter{alg}
    \parbox{0.9\textwidth}{%
\flushright \boxed{\text{Algorithm~\arabic{alg}: Conjugate Gradient Algorithm\label{a:cga}}}
\begin{enumerate}
    \item $\vec x_0$ is the initial guess.
    \item Set $\vec r_0 = \vec b - W \vec x_0$, $\vec p_0= \vec r_0$.
    \item For $k = 1,2,\ldots,n$
    \begin{enumerate}
        \item Compute $\displaystyle a_{k-1} = \frac{\vec r_{k-1}^* \vec r_{k-1}}{\vec r_{k-1}^* W \vec p_{k-1}}$.
        \item Set $\vec x_{k} = \vec x_{k-1} + a_{k-1} \vec p_{k-1}$.
        \item Set $\vec r_{k} = \vec r_{k-1} - a_{k-1} W \vec p_{k-1}$.
        \item Compute $\displaystyle b_{k-1} = - \frac{\vec r_k^* \vec r_k}{\vec r_{k-1}^* \vec r_{k-1}}$.
        \item Set $\vec p_k =  \vec r_k - b_{k-1} \vec p_{k-1}$.
    \end{enumerate}
\end{enumerate}
    }%
}}
\vspace{.05in}

As noted previously, a remarkable fact is that the iterates $\vec x_k$ of the CGA applied to the linear system $W \vec x = \vec b$ are given by the solution of the  minimization problem \eqref{eq:variational} \cite{Hestenes1952}.  From this, we see that $\vec y \in \mathcal K_k$ can be written as
\begin{align*}
    \vec y = \sum_{j=0}^{k-1} c_j W^j \vec b \quad \Rightarrow \quad \vec x - \vec y =  W^{-1} \left(\vec b -  \sum_{j=0}^{k-1} c_j W^{j+1} \vec b\right) = W^{-1} q_{\vec y}(W) \vec b,
\end{align*}
for a polynomial $q_{\vec y}$ of degree at most $k$ and it satisfies $q_{\vec y}(0) = 1$.  Then, computing further,
\begin{align*}
    \|\vec x - \vec y\|_W^2 = \vec b^* q_{\vec y}(W)^* W^{-1} q_{\vec y}(W) \vec b.
\end{align*}
And setting $W = U \Lambda U^*$, we find
\begin{align*}
    \|\vec x - \vec y\|_W^2 = \sum_{j=1}^N \frac{|q_{\vec y} (\lambda_j)|^2}{\lambda_j} |(U^*\vec b)_j|^2 = \int  \frac{|q_{\vec y} (\lambda)|^2}{\lambda} \mu_T( \sd \lambda), \quad T = T(W,\vec b).
\end{align*}
Now, all directional derivatives of this, when $\vec y = \vec x_k$, with respect to coefficients of the polynomial must vanish identically.  This gives a characterization of $q_{\vec x_k}$:  Let $\delta q_k$ be a polynomial of degree at most $k$ that satisfies $\delta q_k(0) = 0$ and we must have
\begin{align*}
    0 = \int q_{\vec x_k}(\lambda) \frac{\delta q_k(\lambda)}{\lambda} \mu_T( \sd \lambda).
\end{align*}
This implies that $q_{\vec x_k}(\lambda)$ is orthogonal to all lower-degree polynomials, with respect to $\mu_T$:  It is given by
\begin{align*}
    q_{\vec x_k}(\lambda) = \frac{\pi_k(\lambda;\mu_T)}{\pi_k(0;\mu_T)}.
\end{align*}
\begin{proposition} \label{prop:errors}
Let $\vec x_k$ be the computed solution at step $k$ of the CGA applied to $W \vec x = \vec b$. For any $k \in \N$, with $T=T(W,\vec b)$,
\begin{align*}
    \|\vec e_k\|_W^2  =  \frac{c_k(0;\mu_{T})}{\pi_k(0;\mu_{T})}
    \quad 
    \text{and}
    \quad
    \|\vec r_k\|^2_2 
    = \frac{\prod_{j=0}^{k-1} b_j(\mu_T)^2}{\pi_k(0;\mu_{T})^2}.
\end{align*}
\end{proposition}
\begin{proof}
By orthogonality
\begin{align*}
\|\vec e_k\|_W^2  &=  \int_{\mathbb R}\frac{\pi_k(\lambda;\mu_{T})^2}{\lambda\pi_k(0;\mu_{T})^2} \mu_T(d \lambda) \\
&= \int_{\mathbb R}\frac{\pi_k(\lambda;\mu_{T}) \left( \pi_k(0;\mu_{T}) \lambda^{-1} +  \sum_{j=1}^k c_j \lambda^{k-1} \right) }{\pi_k(0;\mu_{T})^2} \mu_T(d \lambda)  =  \int_{\mathbb R}\frac{\pi_k(\lambda;\mu_{T})}{\lambda \pi_k(0;\mu_{T})} \mu_T(d \lambda)\\
& = \frac{c_k(0;\mu_{T})}{\pi_k(0;\mu_{T})}.
\end{align*}
For the $\vec r_k$ equation, by definition of the polynomials $\{p_n\}$, we have that
\begin{align}\label{eq:bs}
\int_{\mathbb R}{\pi_k(\lambda;\mu_{T})^2}\mu_T(d \lambda)
=\frac{1}{\ell_k^2}
\int_{\mathbb R}{p_k(\lambda;\mu_{T})^2}\mu_T(d \lambda)
=\frac{1}{\ell_k^2}
=\prod_{j=0}^{k-1} b_j(\mu_T)^2.
\end{align}
\end{proof}

\subsection{MINRES}

The MINRES algorithm, at iteration $k$ gives the solution of
\begin{align*}
    \vec x_k = \mathrm{argmin}_{\vec y \in \mathcal K_k} \|\vec b - W \vec y\|_2.
\end{align*}
More explicitly, the algorithm is given by:

\centerline{\noindent\fbox{%
\refstepcounter{alg}
    \parbox{0.9\textwidth}{%
\flushright \boxed{\text{Algorithm~\arabic{alg}: MINRES Algorithm for $W\vec x = \vec b$\label{a:minres}}}
\begin{enumerate}
    \item Suppose $W = W^* \in \mathbb C^{N \times N}$, $\epsilon > 0.$
    \item Set $\vec q_1 = \vec b/\|\vec b\|_2.$
    \item For $k = 1,2,\ldots,n$, $n \leq N$
    \begin{enumerate}
         \item Compute $\displaystyle a_{k-1} = (W \vec q_k - b_{k-2} \vec q_{k-1})^* \vec q_k$.
        \item Set $\vec v_k = W \vec q_k - a_{k-1} \vec q_k - b_{k-2} \vec q_{k-1}$.
        \item Compute $b_{k-1} = \|\vec v_k\|_2$ and if $b_{k-1} \neq 0$, set $\vec q_{k+1} = \vec v_k/b_{k-1}$.
        \item Form 
        \begin{align*}
 \tilde T_k = \begin{bmatrix} a_0 & b_0 \\
    b_0 & a_1 & \ddots \\
    & \ddots & \ddots & b_{k-2} \\
    & & b_{k-2} & a_{k-1}\\
    &&& b_{k-1}\end{bmatrix}.
    \end{align*}
        \item Compute $\vec z_k = \mathrm{argmin}_{\vec z \in \mathbb C^k} \| \tilde T_k \vec z - \|\vec b\|_2 \vec f_1 \|_2.$
        \item If $\| \tilde T_k \vec z_k - \|\vec b\|_2 \vec f_1 \|_2 < \epsilon$, return $\vec x_k = \begin{bmatrix} \vec q_1 & \cdots & \vec q_k \end{bmatrix} \vec z_k$.
    \end{enumerate}
\end{enumerate}
    }%
}}
\vspace{.05in}

Following the same prescription as in the previous section we are led to the problem of finding the polynomial $r_{\vec x_k}$ of degree less than or equal to $k$ satisfying $r_{\vec x_k}(0) = 1$ that minimizes
\begin{align*}
     \|\vec b - W \vec y\|_2^2 = \sum_{j=1}^N |r_{\vec y} (\lambda_j)|^2 |(U^*\vec b)_j|^2 = \int  |r_{\vec y} (\lambda)|^2 \mu_T( \sd \lambda), \quad T = T(W,\vec b),
\end{align*}
among all such polynomials. We then must have
\begin{align*}
    0 = \int r_{\vec x_k}(\lambda) \delta r_k(\lambda) \mu_T(\sd \lambda)
\end{align*}
for all polynomials $\delta r_k$ of degree less than or equal to $k$ with $\delta r_k(0) = 0$.  So, write $r_{x_k}(\lambda) = \sum_{j = 0}^{k} c_j p_j(\lambda;\mu_T)$.  And choosing $\delta r_k(\lambda) = p_\ell(\lambda;\mu_T) - p_\ell(0;\mu_T)$ we find
\begin{align*}
    0 = \int \left( \sum_{j = 0}^{k} c_j p_j(\lambda;\mu_T)\right) \left( p_\ell(\lambda;\mu_T) - p_\ell(0;\mu_T)\right) \mu_T(\sd \lambda) \Leftrightarrow c_\ell = p_\ell(0;\mu_T) c_0.
\end{align*}
From this, we obtain
\begin{align}\label{eq:rxk}
    r_{\vec x_k}(\lambda;\mu_T) = \frac{\sum_{j=0}^k p_j(0;\mu_T) p_j(\lambda;\mu_T) }{ \sum_{j=0}^k p_j^2(0;\mu_T)}.
\end{align}
\begin{proposition}\label{prop:MINRES}
  Let $\vec x_k$ be the computed solution at step $k$ of the MINRES algorithm applied to $W\vec x= \vec b$. For any $k \in \N$, with $T=T(W,\vec b)$
  \begin{align*}
      \|\vec r_k\|_2^2 &= \frac{1}{\sum_{j=0}^k p_j^2(0;\mu_T)}\\
      & = \frac{1}{b_k(\mu_T)^2 \left[ p_{k+1}'(0;\mu_T) p_k(0;\mu_T) - p_k'(0;\mu_T) p_{k+1}(0;\mu_T) \right]},\\
      & = \frac{\prod_{j=0}^{k-1} b_j(\mu_T)^2}{\pi_{k+1}'(0) \pi_k(0;\mu_T) - \pi_k'(0;\mu_T) \pi_{k+1}(0;\mu_T)}
  \end{align*}
\end{proposition}
\begin{proof}
Integrating \eqref{eq:rxk}
\begin{align*}
    \|\vec b - W \vec x_k\|_2^2 = \frac{1}{\sum_{j=0}^k p_j^2(0;\mu_T)}.
\end{align*}
Employing the Christoffel-Darboux formula,
\begin{align*}
    \sum_{j=0}^k p_j^2(0;\mu_T) &= \frac{\ell_k}{\ell_{k+1}} \left[ p_{k+1}'(0;\mu_T) p_k(0;\mu_T) - p_k'(0;\mu_T) p_{k+1}(0;\mu_T) \right] \\
    &= b_k(\mu_T)^2 \left[ p_{k+1}'(0;\mu_T) p_k(0;\mu_T) - p_k'(0;\mu_T) p_{k+1}(0;\mu_T) \right].
\end{align*}
Then using \eqref{eq:bs}
\begin{align*}
    p_k(\lambda;\mu_T) =\left( \prod_{j=0}^{k-1} b_j(\mu_T)^{-1} \right)\pi_k(\lambda;\mu_T)
\end{align*}
we find the alternate expression
\begin{align*}
    \sum_{j=0}^k p_j^2(0;\mu_T) = \left( \prod_{j=0}^{k-1} b_j(\mu_T)^{-2} \right)\left[ \pi_{k+1}'(0;\mu_T) \pi_k(0;\mu_T) - \pi_k'(0;\mu_T) \pi_{k+1}(0;\mu_T) \right].
\end{align*}
\end{proof}

\subsection{The CGA on the normal equations}

Next, for $X \in \mathbb C^{N\times M}$, $N \leq M$, consider solving the normal equations $XX^* \vec x = X \vec b$ with the CGA. The appearance of $X$ on the right-hand side changes the minimization problem one has to consider.  With $W = XX^*$, the CGA will solve
\begin{align*}
    \vec x_k = \mathrm{argmin}_{\vec y \in \mathcal K_k} \|\vec x - \vec y\|_W, \quad \mathcal K_k = \span\{X \vec b, W X \vec b, \ldots, W^{k-1} X \vec b\}.
\end{align*}
As before, we express
\begin{align*}
    \vec x - \vec y = W^{-1} q_{\vec y}(W) X \vec b.
\end{align*}
Using the singular value decomposition $ X = U \Sigma V^*$ where $U,V$ are square matrices, we write
\begin{align*}
    \| \vec x - \vec y \|_W^2 &= \vec b^* V \Sigma^* U^* q_{\vec y}(W)^*W^{-1} q_{\vec y}(W) U \Sigma V^* \vec b,\\
    & = \vec b^* V \Sigma^* q_{\vec y}(\Lambda)^*\Lambda^{-1} q_{\vec y}(\Lambda) \Sigma V^* \vec b
\end{align*}
where $\Lambda = \Sigma \Sigma^*$.  Since $\Sigma$ has its last $M -N$ columns being identically zero, we use the notation $\Sigma = \begin{bmatrix} \Sigma_0 & 0 \end{bmatrix}$ and find $\Lambda = \Sigma_0^2.$  Thus
\begin{align*}
    \| \vec x - \vec y \|_W^2 = \vec c^* \Sigma_0 q_{\vec y}(\Lambda)^*\Lambda^{-1} q_{\vec y}(\Lambda) \Sigma_0 \vec c, \quad \vec c = \begin{bmatrix} I & 0 \end{bmatrix} V^* \vec b.
\end{align*}
The techniques used in the case of MINRES directly apply.
\begin{proposition}\label{prop:CGANE}
  Let $\vec x_k$ be the computed solution at step $k$ of applying the CGA to the normal  equations $XX^*\vec x= X\vec b$, $X \in \mathbb C^{N \times M}$, $N \leq M$. For any $k \in \N$,
  \begin{align*}
      \|\vec e_k\|_W^2 & = \frac{\prod_{j=0}^{k-1} b_j(\nu)^2}{\pi_{k+1}'(0;\nu) \pi_k(0;\nu) - \pi_k'(0;\nu) \pi_{k+1}(0;\nu)} = \frac{1}{\sum_{j=0}^k p_j^2(0;\nu)},
  \end{align*}
  where
  \begin{align}\label{eq:nu}
      \nu = \sum_{j=1}^N \omega_j \delta_{\lambda_j},  \quad \omega_j = |(V^*\vec b)_j|^2,
  \end{align}
  $X = U \Sigma V^*$ is the singular value decomposition of $X$ and $\lambda_1,\ldots,\lambda_N$ are the eigenvalues of $XX^*$.
\end{proposition}

\section{Universality}\label{sec:univ}

\subsection{Bidiagonal central limit theorem, Gaussian case}


Throughout the asymptotic analysis that follows $d$ will be a fixed positive real number and $ \mfd = N/M \To[M] d$.
Taking the entrywise limit in \eqref{eq:GK}, using the notation
\begin{align}
    \frac{1}{\sqrt{\beta M}} X_{\mathrm {GK}} &= \begin{bmatrix} H & 0\end{bmatrix},\notag\\
    H & \lawequals \frac{1}{\beta M}\left[\begin{array}{ccccc} \chi_{\beta M}  \\
    \chi_{\beta(N-1)} & \chi_{\beta(M-1)}  \\
    & \chi_{\beta(N-2)} & \chi_{\beta(M-2)}  \\
    && \ddots & \ddots \\
    &&& \chi_{\beta} & \chi_{\beta(M - N + 1)} \end{array}\right], \label{eq:dist_H}
\end{align}
it follows that
\begin{align*}
    H \Wkto[N] \mathbb H_{d} = \begin{bmatrix} 1 & \\
    \sqrt{d} & 1\\
    & \sqrt{d} & 1\\
     && \ddots & \ddots
    \end{bmatrix}.
\end{align*}
This limit is in the sense of weak convergence of the finite-dimensional marginals of a random infinite bidiagonal matrix.

Furthermore, for a $\chi_k$ random variable
\[
{\chi_k} - \sqrt{k} \Wkto[k] \mathcal N_1 (0,1/2),
\]
and so by independence, for iid standard normals $\{Z_j\}_{1}^\infty,$
\begin{equation}\label{eq:TriCLT}
    \sqrt{2\beta M}(H - \mathbb{H}_{\mfd})
    \Wkto[N] \mathbb{G} = 
    \begin{bmatrix} Z_1 & \\
    Z_2 & Z_3\\
    & Z_4 & Z_5\\
     && \ddots & \ddots
    \end{bmatrix}.
\end{equation}
From here, it follows immediately that the Jacobi matrix produced by the Lanczos algorithm applied to $\mathcal W_{\beta}(N,M)$ has a limit, in the same sense of finite-dimensional marginal convergence, to an infinite tridiagonal matrix. 

\begin{definition}
Given a positive-definite Jacobi matrix $T$ we define $\varphi$ to be the function that gives the Cholesky factorization of $T$.  That is $\varphi(T) = H$ where $H$ is a lower-triangular bidiagonal matrix with all non-negative entries and $HH^* = T$.
\end{definition}
The Cholesky factorization $\varphi(T)$ is unique for $T > 0$ and $\varphi$ is generically differentiable (see \cite{Edelman2005}). The actual algorithm to compute it is given as follows:\\

\centerline{\noindent\fbox{%
\refstepcounter{alg}
    \parbox{0.9\textwidth}{%
\flushright \boxed{\text{Algorithm~\arabic{alg}: Jacobi matrix Cholesky factorization \label{a:chol}}}
\begin{enumerate}
    \item Suppose $T$ is an $N \times N$ positive-definite Jacobi matrix, set $H = T.$
    \item For $k = 1,2,\ldots,N-1$
    \begin{enumerate}
        \item Set $H_{k+1,k+1} = H_{k+1,k+1} - \displaystyle\frac{H_{k+1,k}^2}{H_{kk}}.$
        \item Set $H_{k:k+1,k} = H_{k:k+1,k}/\sqrt{H_{k,k}}.$
    \end{enumerate}
    \item Set $H_{N,N} = \sqrt{H_{N,N}}.$
    \item  Return $\varphi(T) = H.$
\end{enumerate}
    }%
}}
\vspace{.05in}

The following is immediate.

\begin{proposition}\label{prop:GaussianBidiagonalLimit}
Let $W \lawequals \mathcal W_{\beta}(N,M)$, $N \leq M$.  For any sequence of unit vectors $\vec b= \vec b_N$ of length $N,$
\[
\sqrt{2\beta M}(\varphi(T(W, \vec b_N)) - \mathbb H_{\mfd}) \Wkto[N] \mathbb G.
\]
\end{proposition}
Now, define
\begin{align*}
    \mathbb H_d \mathbb H_d^* = \mathbb T_d := \begin{bmatrix} 1 & \sqrt{d} \\
    \sqrt{d} & 1 + d & \sqrt{d}\\
    & \sqrt{d} & 1+ d  & \ddots&\\
     && \ddots & \ddots
    \end{bmatrix}.
\end{align*}

\begin{proposition}\label{prop:GaussianPolyFCLT}
  Let $W \lawequals \mathcal W_{\beta}(N,M)$ for $N\leq M$ where $\mfd  \To[M] d \in (0,1]$.   Then for any sequence of unit vectors $\vec b= \vec b_N$ of length $N,$ with $T=T(W,\vec b),$ the vector
  \begin{align*}
     \left( \sqrt{\beta M}\vec f_1^* (T^k - \mathbb{T}_{\mfd}^k) \vec f_1 \right)_{k \geq 1} = \left( \sqrt{\beta M}\int_{\mathbb R} x^k (\mu_T(\sd x) - \varrho_{\mfd}(\sd x))\right)_{k \geq 1},
  \end{align*}
  converges in the sense of finite-dimensional marginals to a centered Gaussian random vector $\mathcal G = (G_1)_{k\geq 1}$.
\end{proposition}
\begin{proof}
The equality follows using \eqref{eq:TriMoment}.  The proposition then follows using \eqref{eq:TriCLT} because, for each $k$, $\sqrt{M}\vec f_1^* (T^k - \mathbb{T}_\mfd^k) \vec f_1$ depends only on a finite number of elements of $T$.
\end{proof}

\subsection{Contour integral reformulation of the moments}
Let $\Gamma$ be a simple curve that encloses the nonzero spectrum of a symmetric tridiagonal matrix $T$.  Then
\begin{align*}
    m_k(\mu_T) = \frac{1}{2 \pi i} \oint_{\Gamma} z^k c_0(z;\mu_T) \sd z.
\end{align*}
Now, let $\Gamma=\Gamma_d$ be a smooth simple contour that properly encloses the support of the Marchenko--Pastur law \eqref{eq:MP}.  

We denote the Stieltjes transform $s_d(z)$ of \eqref{eq:MP} by
\begin{equation}
  s_d(z) 
  = \int_\R \frac{ \varrho_d(\sd \lambda)}{\lambda-z}.
  \label{eq:smp}
\end{equation}
There are many classical references for the following result.
\begin{theorem}[Global eigenvalue bounds, see, e.g. \cite{Davidson2001,Geman1980,Silverstein1985,Vershynin2009}]\label{t:geb}
For the eigenvalues $\lambda_N \leq \cdots \leq \lambda_1$ of $W \lawequals \mathcal W_{\beta}(N,M)$, $N \leq M$ and $t > 0$
\begin{align*}
    \mathbb P \left( 1 - \sqrt{\frac{N}{M}} - t \leq \lambda_N^{1/2} \leq \lambda_1^{1/2} \leq 1 + \sqrt{\frac{N}{M}} + t\right) \To[M] 0.
\end{align*}
\end{theorem}

\noindent Hence with probability tending to $1$ as $M \to \infty$, $\mfd \To[M] d$, the support of $\mu_T$, $T = T(W,\vec b)$ is contained within $\Gamma_d$.
As a corollary, we have:
\begin{corollary}\label{cor:GaussianPolyFCLT}
  Let $W \lawequals \mathcal W_{\beta}(N,M)$ for $N \leq M$ where $\mfd \To[M] d \in (0,1]$.  Then for any sequence of unit vectors $\vec b= \vec b_N$ of dimension $N,$ with $T=T(W,\vec b),$ the vector
  \begin{align*}
  \left( \frac{\sqrt{M}}{2 \pi i}\oint_{\Gamma_d} z^k (c_0(z; \mu_T) - s_\mfd(z))\sd z\right)_{k \geq 1}
  \Wkto[N] \mathcal{G},
  \end{align*}
  in the sense of finite-dimensional marginals, where $\mathcal{G}$ is the same process as in Proposition \ref{prop:GaussianPolyFCLT}.
\end{corollary}

We also need to treat the case of $k = -1$. Suppose $T = H H^T$ where $H$ is real, square, lower-triangular and given by
\begin{align}\label{eq:Hdef}
   H = \begin{bmatrix} \alpha_0 \\
   \beta_0 & \alpha_1 \\
   & \beta_1 & \alpha_2 \\
   && \ddots & \ddots
   \end{bmatrix}
\end{align} 
Then $T_{11} =\alpha_0^2$ and $T_{jj} = \alpha_{j-1}^2 + \beta_{j-2}^2$ for $j > 1$.  Let $\tilde H$ be the matrix formed by removing the first row and column of $H$ and let $\tilde T = \tilde H \tilde H^T$.  Then it follows by Cramer's rule that
\begin{align}\label{eq:Cramer}
    \vec f_1^* T^{-1} \vec f_1 = \frac{ \det(\beta_0^2 \vec f_1 \vec f_1^* +  \tilde T)}{\det T} = \frac{ (\det \tilde T) (1 + \beta_0^2 \vec f_1^* \tilde T^{-1} \vec f_1 )}{\det T} = \frac{1}{\alpha_0^2} ( 1 + \beta_0^2\vec f_1^* \tilde T^{-1} \vec f_1 ).
\end{align}
From this expression, one obtains
\begin{align*}
    \vec f_1^* T^{-1} \vec f_1 = \frac{1}{\alpha_0^2} \left( 1 + \sum_{j=1}^{N-1} \prod_{k=1}^j \left(  \frac{\beta_{k-1}^2}{\alpha_k^2} \right)\right).
\end{align*}
Following \cite[Theorem 3.2.12]{Muirhead1982}:
\begin{proposition}\label{p:Muir}
  Let $\vec y$ be random vector in $\mathbb C^N$ that does not vanish a.s.  Let $X$ be an $N \times M$ matrix with independent $\mathcal N_\beta(0,1)$ entries independent of $\vec y.$  Then
  \begin{align*}
      \frac{\vec y^* \vec y}{\vec y (XX^*)^{-1} \vec y} 
      \lawequals \beta^{-1}\chi_{\beta(M - N + 1)}^2
  \end{align*}
  and therefore
  \begin{align*}
      \vec f_1^* T^{-1} \vec f_1 \lawequals \frac{\beta M}{\chi_{\beta(M - N +1)}^2}.
  \end{align*}
\end{proposition}
\begin{proof}
The first claim can be established using the QR factorization of $X$. 
The second claim for $\vec f_1^* T^{-1} \vec f_1$ follows from the first once we realize $T = T(M^{-1}X X^*, \vec b),$ then $\vec f_1^* T^{-1} \vec f_1 = M\vec b^* (XX^*)^{-1} \vec b$. 
 \end{proof}
 As we can also apply the same proposition to an $(N-1)\times (M-1)$ matrix of normals, and conclude
\begin{align*}
    \vec f_1^* \tilde T^{-1} \vec f_1 \lawequals \frac{\beta M}{\chi_{\beta(M - N +1)}^2}.
\end{align*}
Using \eqref{eq:Cramer} this provides a remarkable identity in law involving chi-square distributions:
\begin{proposition}
  For any integers $\ell \geq 0$ and $M \geq N \geq 1$
  \begin{align*}
      \frac{\beta M}{\chi^2_{\beta(M - N +1)}} \lawequals \frac{1}{\chi_{\beta(M - \ell)}^2} \left(1 + \chi_{\beta(N - \ell-1)}^2 \frac{\beta M}{\chi_{\beta(M-N+1)}^2}  \right)
  \end{align*}
  where the chi-squared variables on the right-hand side are mutually independent.
\end{proposition}
But more importantly, iterating \eqref{eq:Cramer} $\ell$ times and applying and using Proposition \ref{p:Muir} to describe the remainder, we have:
\begin{proposition}\label{p:finite}
    Suppose $H$ is distributed as in \eqref{eq:dist_H}.  Then for $0 < \ell < N$
    \begin{align*}
         \vec f_1^* T^{-1} \vec f_1 = \frac{1}{\alpha_0^2} \left( 1 +  \sum_{j=1}^{\ell} \prod_{k = 1}^j \frac{\beta_{k-1}^2}{\alpha_{k}^2} + \left(\beta_\ell^2  \frac{\beta M}{\chi_{\beta(M -N +1)}^2} \right)\prod_{k = 1}^{\ell} \frac{\beta_{k-1}^2}{\alpha_{k}^2}\right)
    \end{align*}
    where $\chi_{\beta(M -N +1)}$ depends only on $H_{\ell+1:N,\ell+1:N}$.
\end{proposition}

The following notation is convenient.
\begin{definition}
We write $X_M = c_M + Y_M + o(M^{-1/2})$ if
\begin{align*}
    \sqrt{M}(X_M - c_M), \quad\text{and}\quad \sqrt{M}(Y_M),
\end{align*}
converge, in distribution, to the same distribution as $M \to \infty$.  
\end{definition}

Let $\ell$ be fixed. We use the approximation in distribution \eqref{eq:TriCLT}, $\alpha_j = 1 + Z_{2j+1} /\sqrt{2 \beta M} + o(M^{-1/2})$, $\beta_j = \sqrt \mfd +  Z_{2j + 2}/\sqrt{2 \beta M}+ o(M^{-1/2})$ to find
\begin{align}\label{eq:ajs}
    \alpha_j^2 &= 1 + \frac{\sqrt{2}}{\sqrt{\beta M}}Z_{2j + 1} + o(M^{-1/2}),\\\label{eq:bjs}
    \beta_j^2 &= \mfd + \frac{\sqrt{2\mfd}}{\sqrt{\beta M}}Z_{2j + 2} + o(M^{-1/2}),\\
    \frac{\beta M}{\chi_{\beta(M -N +1)}^2} & = \frac{1}{1-\mfd} \left( 1 + \frac{\sqrt{2}}{\sqrt{1-\mfd} \sqrt{\beta M}}  Z_0\right) + o(M^{-1/2}), \quad Z_0 \lawequals \mathcal N_1 (0,1),\notag
\end{align}
and compute as $N \to \infty$
 \begin{align*}
     1 +& \sum_{j=1}^{\ell} \prod_{k=1}^j \left(  \frac{\beta_{k-1}^2}{\alpha_k^2} \right)   =   1  + \sum_{j=1}^{\ell} \prod_{k =1}^{j}  \mfd\left[1  +  \frac{\sqrt{2}}{\sqrt{\beta M}} \left(  Z_{2k}/\sqrt{d}  - Z_{2k+1} \right) \right]\\
     & = \frac{1 - \mfd^{\ell + 1}}{1-\mfd} +  \frac{\sqrt{2}}{\sqrt{\beta M}} \sum_{k =1}^{\ell}  \mfd^k \frac{1 - \mfd^{\ell-k+1}}{1 -\mfd} \left( Z_{2k}/\sqrt{\mfd}  - Z_{2k+1} \right) + o(M^{-1/2}).
 \end{align*}
Thus
 \begin{align*}
    \sqrt{\beta M} \left( \vec f_1^* T^{-1} \vec f_1  - \frac{1}{1 -\mfd} \right) \Wkto[M] \sqrt{2} \frac{Z_1}{d-1} + \sqrt{2} \sum_{k = 1}^\infty \frac{d^k}{1-d} (Z_{2k}/\sqrt{d} - Z_{2k +1}).
 \end{align*}
We arrive at the following proposition.

\begin{proposition} \label{prop:dist-limit}
  Suppose $H$ is distributed as in \eqref{eq:dist_H} where the entries are labelled according to \eqref{eq:Hdef} and $T = HH^*$. Let $\mathcal Z = [Z_1,Z_2,\ldots]^T$ be a vector of iid standard normal random variables. Then if $N \leq M, \mfd \To[M] d \in (0,1)$
  \begin{align*}
          \sqrt{\beta M} &\left( \vec f_1^* T^{-1} \vec f_1  - \frac{1}{1-\mfd} \right) \Wkto[M] Z_{-1},\\
          Z_{-1} &: = -\sqrt{2} \frac{Z_1}{1-d} + \sqrt{2} \sum_{k = 1}^\infty \frac{d^k}{1-d} (Z_{2k}/\sqrt{d} - Z_{2k +1}).
  \end{align*}
Additionally,
  \begin{align*}
      \sqrt{\beta M} \left( \begin{bmatrix} \vec f_1^* T^{-1} \vec f_1 \\ \alpha_0 \\ \beta_0 \\ \alpha_1 \\ \beta_1 \\\vdots \end{bmatrix} - \begin{bmatrix} \frac{1}{1-\mfd} \\ 1 \\ \sqrt{\mfd} \\ 1 \\ \sqrt{\mfd} \\ \vdots   \end{bmatrix} \right) \Wkto[M] \begin{bmatrix} Z_{-1} \\ \mathcal Z/\sqrt{2} \end{bmatrix}
  \end{align*}
  in the sense of convergence of finite-dimensional marginals.
\end{proposition}

\subsection{Universality for the moment fluctuations of the spectral measure}

We now generalize Corollary \ref{cor:GaussianPolyFCLT} to general distributions. 
Let $R(z) = R(z;X) = (XX^* - z\Id)^{-1}$ denote the resolvent of $XX^*$ and define  $ G(z) = G(z;X) = \begin{bmatrix} -I & X^* \\ X & -zI \end{bmatrix}^{-1}$.  The following is a direct consequence of \cite[Theorems 3.6 and 3.7]{Knowles2017}.

\begin{proposition} \label{prop:tail}
  Suppose $X$ is a sample covariance matrix with $\mfd = N/M \To[M] d \in (0,\infty)$.  For any $\delta,\epsilon > 0$ and for any $R, D>0$ there is a constant $C$ so that
  for all $M \in \N,$
  \begin{align*}
    \sup_{z \in \Gamma}&
    \sup_{\vec v, \vec w \in \C^{N+M}}
    \Pr\left[ 
      \left| \vec v^* G(z) \vec w
      - \vec v^*\Pi_\mfd(z)  \vec w 
      \right|
      \geq \|\vec v\|\|\vec w\| M^{\epsilon - 1/2}
    \right] \leq C M^{-D}, \\
    \quad \Pi_\mfd(z) &= \begin{bmatrix} - (1 + s_\mfd (z))^{-1} I_M & 0 \\ 0 & s_\mfd(z)I_N \end{bmatrix},
  \end{align*}
  and therefore
  \[
    \sup_{z \in \Gamma}
    \sup_{\vec v, \vec w \in \C^N}
    \Pr\left[ 
      \left| \vec v^* R(z) \vec w
      - s_\mfd(z) \vec v^* \vec w 
      \right|
      \geq \|\vec v\|\|\vec w\| M^{\epsilon - 1/2}
    \right] \leq C M^{-D},
  \]
  where $\Gamma$ is any bounded simple closed curve that does not intersect the support of $\varrho_d$.
\end{proposition}

Define the classical eigenvalue locations $\gamma_i = \gamma_i^{(N)}$ by $N \int_{\gamma_i}^\infty \varrho_{\mfd}(\sd x) = i -1/2$ and from \cite[Theorem 3.12]{Knowles2017} we have:
\begin{proposition}[Eigenvalue rigidity]\label{p:rigidity}
  Let $X$ be a sample covariance matrix and denote the eigenvalues of $XX^*$ by $\lambda_1 \geq \lambda_2 \geq \cdots \geq \lambda_N$.  For any $\epsilon > 0$ and for any $D>0$ there is a constant $C$ so that
  \begin{align*}
  \Prob \left(  |\lambda_i  -\gamma_i| > N^{\epsilon}    (\max\{i, N + 1 - i\})^{-1/3} N^{-2/3} ~~\text{for any}~~i\right) \leq C N^{-D}.
  \end{align*}
\end{proposition}

\begin{definition}\label{d:testfun}
Let $\Phi: \mathbb C^n \to \mathbb R$ be bounded.  Suppose, in addition, that for any multi-index $\alpha = (\alpha_1,\ldots,\alpha_n)$, $1 \leq |\alpha| \leq 5$ and for any $\epsilon' > 0$ sufficiently small, we have
  \begin{align*}
      \max \{|\partial^\alpha \Phi(x_1,\ldots,x_n)| : \max_j |x_j| &\leq M^{\epsilon'} \} \leq M^{C_0 \epsilon'},
  \end{align*}
  for $C_0 > 0$.  Then $\Phi$ is called an admissible test function.
\end{definition}

\begin{theorem}[Comparison]\label{t:GC}
   Let $W = XX^*$ and $\tilde W = YY^*$ be two sample covariance matrices such that
  \begin{align*}
    \Exp (\Re X_{ij})^\ell (\Im X_{ij} )^p = \Exp (\Re Y_{ij})^\ell (\Im Y_{ij} )^p, \quad \ell + p \leq 4, \quad 1 \leq i \leq N, \quad 1 \leq j \leq M.
\end{align*}
  For each $j$, let $\Gamma_j = \partial \Omega_j$, $\Omega_j = \overline \Omega_j$ be a simple smooth positively-oriented curve that is uniformly bounded away from support of the Marchenko--Pastur law $\varrho_d$. Suppose that $f_1,f_2,\ldots,f_n$ is a finite collection of functions that are analytic in a neighborhood of $\Omega$.  Then for any admissible test function $\Phi: \mathbb C^n \to \mathbb R$ we have for $T=T(W,\vec b), \tilde T =T(\tilde W,\vec b),$
  \begin{align*}
      \left|\Exp \Phi\left(\frac{\sqrt{M}}{2 \pi i}\oint_{\Gamma_1} f_1(z) (c_0(z; \mu_T) - s_\mfd(z))\sd z,\ldots, \frac{\sqrt{M}}{2 \pi i}\oint_{\Gamma_n} f_n(z) (c_0(z; \mu_T) - s_\mfd(z)) \sd z \right) \right.\\
      -  \left.\Exp\Phi\left(\frac{\sqrt{M}}{2 \pi i}\oint_{\Gamma_1} f_1(z) (c_0(z; \mu_{\tilde T}) - s_\mfd(z)) \sd z,\ldots, \frac{\sqrt{M}}{2 \pi i}\oint_{\Gamma_n} f_n(z) (c_0(z; \mu_{\tilde T}) - s_\mfd(z)) \sd z \right)\right|\\
      \leq C M^{-\sigma}
  \end{align*}
  for some $C,\sigma > 0$.  Here $C$ will depend on $n$, the constants $C_p$ in Definition~\ref{def:SCMs}, $\Phi$, $\Gamma_1,\ldots,\Gamma_n$ and $f_1,\ldots,f_n$ and $\sigma$ will depend on the constant $C_0$ in Definition~\ref{d:testfun}.
\end{theorem}
\begin{remark}
Note that in Theorem~\ref{t:GC}, if $\mfd$ is bounded uniformly away from one, a contour $\Gamma_j$ could just encircle $z = 0$.  And if  $\mfd  \to d \in (0,1]$ the only non-trivial case is where the contour $\Gamma_j$ encircles the entire support of $\varrho_d$.
\end{remark}

This gives immediate corollaries.
\begin{corollary}\label{cor:UniversalMoments}
Suppose $W$ is a sample covariance matrix satisfying the moment matching  condition (Definition~\ref{def:moment}) with $\mfd = N/M \To[M] d \in (0,\infty)$.  Then for any sequence of unit vectors $\vec b= \vec b_N$ of length $N,$ with $T=T(W,\vec b),$ the vector
  \begin{align*}
  \left( \frac{\sqrt{M}}{2 \pi i}\oint_{\Gamma} z^k (c_0(z; \mu_T) - s_\mfd (z)) \sd z\right)_{k \geq 1}
  \Wkto[N] \mathcal{G},
  \end{align*}
  in the sense of finite-dimensional marginals, where $\mathcal{G}$ is the same process as in Proposition \ref{prop:GaussianPolyFCLT}. 
\end{corollary}

\begin{corollary}\label{cor:UniversalProcess}
  Suppose $W$ is a sample covariance matrix satisfying the moment matching  condition (Definition~\ref{def:moment}) with $\mfd = N/M \To[M] d \in (0,\infty)$.  Then for any sequence of unit vectors $\vec b= \vec b_N$ of length $N,$ with $T=T(W,\vec b)$, let $H$ be given by the Cholesky factorization of of $T$, $H = \varphi(T)$ and label the entries of $H$ as in \eqref{eq:Hdef}.  Then Proposition~\ref{prop:dist-limit} holds for $H$.
\end{corollary}
\begin{proof}
Fix $k$.  For all $N > k,$ the Hankel matrix of moments $(m_{j+r-2}(\mu_T))_{j,r=1}^{k}$ is positive definite almost surely.  On this set, the mapping to $(m_j(\mu_T) : 0 \leq j \leq 2k) \mapsto T_k(W,\vec b)$ is differentiable.  It follows that $H_k$, the upper-left $k\times k$ subblock of $H$ is also a differentiable function $m_j(\mu_T)$, $j = 0,1,\ldots,2k$.  Then the corollary follows directly from Theorem~\ref{t:GC}.
\end{proof}

Before we prove Theorem~\ref{t:GC}, we establish some intermediate results.

\begin{lemma}
For an $N \times M$ matrix $X$ and $\Im z \neq 0$
\begin{align}\label{eq:schur}
    \begin{bmatrix} -I & X^* \\ X & -zI \end{bmatrix}^{-1} &= \begin{bmatrix}
        (z^{-1}XX^*-I)^{-1} & (X^*X-zI)^{-1}X^*\\
         X(X^*X-zI)^{-1}& (XX^*- zI)^{-1},\\
    \end{bmatrix}\\ \label{eq:schur_norm}
    \left\|\begin{bmatrix} -I & X^* \\ X & -zI \end{bmatrix}^{-1} \right\| &\leq (|z| + 1) |\Im z|^{-1} + 2 \sqrt{ |\Im z|^{-1} + |z| |\Im z|^{-2}}
\end{align}
\end{lemma}

Recall that $\vec f_1,\vec f_2, \ldots$ denotes the standard basis and we use the notation $ \hat {\vec u} = \begin{bmatrix}0 \\  \vec u \end{bmatrix} \in \mathbb C^{N + M}$ for $\vec u \in \mathbb C^N$.

\begin{lemma}[Resolvent expansion with leading-order correction] \label{l:resolve}
  Let $X$ be an iid matrix satisfying the assumptions of Definition~\ref{def:SCMs}. Let $Q$ be the matrix that is equal to $X$ with the exception of one entry that is set to zero so that $X = Q + X_{ij} \vec f_{i} \vec f_{j}^*$ for some $1 \leq i \leq N, 1 \leq j \leq M$.  For two unit vectors $\vec u, \vec v \in \mathbb C^N$
  \begin{align*}
    \hat {\vec u}^* S(z;X) \hat {\vec v} &= \hat {\vec u}^* S(z;Q) \hat {\vec v} + \sum_{k=1}^{3} M^{-k/2} J_k + M^{-5/2}J_4,\\
    S(z;X) &= \sqrt{M} \left[ G(z;X) - \Pi_d(z)\right],
\end{align*}
and for every $\epsilon > 0$ and $D > 0$ there exists $C > 0$ such that $J_4$ satisfies
\begin{align*}
    \mathbb P( |J_4| > M^\epsilon ) \leq C M^{-D}.
\end{align*}
In addition, $J_k$ for $k < 4$ is a finite sum of the form
\begin{align*}
    J_k = \sum_{\ell} f_{k,\ell} g_{k,\ell},
\end{align*}
where $g_{k,\ell}$ is a monomial in $ X_{ij}\sqrt{M}$ and $\overline{X_{ij}}\sqrt{M}$ with degree at most $k+1$ and $f_{k,\ell}$ is independent of $X_{ij}$ satisfying
that for every $\epsilon > 0$ and $D > 0$ there exists $C > 0$ such that
\begin{align*}
    \mathbb P( |f_{k,\ell}| > M^\epsilon ) \leq C M^{-D}.
\end{align*}
\end{lemma}

\begin{proof}
Write $V := X_{ij} \vec f_{i+M} \vec f_{j}^*  +  \overline{X_{ij}} \vec f_{j} \vec f_{i+M}^*$.   Consider for a diagonal matrix $D$
\begin{align}
    \hat {\vec u}^* DVD \hat {\vec v} &= X_{ij} (\hat {\vec u}^*D\vec f_{i+M})(\vec f_{j}^*D\hat {\vec v}) + \overline{X_{ij}} (\hat {\vec u}^*D\vec f_{j}) (\vec f_{i+M}^*D\hat {\vec v}) \notag \\
   &= X_{ij} (\hat {\vec u}^*D_{i+M,i+M}\vec f_{i+M})(\vec f_{j}^*D_{jj}\hat {\vec v}) + \overline{X_{ij}} (\hat {\vec u}^*D_{jj}\vec f_{j}) (\vec f_{i+M}^*D_{i+M,i+M}\hat {\vec v}) = 0.\label{eq:vanish}
\end{align}
This is because $1\leq j \leq M$ and $\hat {\vec u}, \hat {\vec v}$ must have zeros in their first $M$ entries.

We then consider the expansion of
\begin{align*}
    S(z;X) &= S(z;Q) - \sqrt{M} G(z;Q)V G(z;Q) + \cdots\\
    & + \sqrt{M}[G(z;Q)V]^4 G(z;Q)\\
    & - \sqrt{M}[G(z;Q)V]^{5}G(z;X).
\end{align*}
We write
\begin{align*}
    \hat {\vec u}^* &\sqrt{M} G(z;Q)V G(z;Q) \hat {\vec v} 
    =  \sqrt{M} \hat {\vec u}^*\Pi_d(z) V\Pi_d(z) \hat {\vec v}  + E_{ij}.
\end{align*}
From \eqref{eq:vanish} the first term vanishes.  Explicitly,
\begin{align*}
    E_{ij} &=  \hat {\vec u}^* \left[ S(z;Q)V G(z;Q) + G(z;Q)V S(z;Q) + M^{-1/2} S(z;Q)V S(z;Q) \right] \hat {\vec v},\\
    & =  X_{ij} (\hat {\vec u}^* S(z;Q) \vec f_{j+M})( \vec f_i^* G(z,Q) \hat {\vec v}) + \overline{X_{ij}} ( \hat {\vec u}^* G(z;Q) \vec f_{i})( \vec f_{j+M}^* S(z,Q) \hat {\vec v})\\
    & + X_{ij} (\hat {\vec u}^* G(z;Q) \vec f_{j+M})( \vec f_i^* S(z,Q) \hat {\vec v}) + \overline{X_{ij}} ( \hat {\vec u}^* S(z;Q) \vec f_{i})( \vec f_{j+M}^* G(z,Q) \hat {\vec v})\\
    & +  M^{-1/2}X_{ij} (\hat {\vec u}^* S(z;Q) \vec f_{j+M})( \vec f_i^* S(z,Q) \hat {\vec v}) + M^{-1/2}\overline{X_{ij}} ( \hat {\vec u}^* S(z;Q) \vec f_{i})( \vec f_{j+M}^* S(z,Q) \hat {\vec v}).
\end{align*}
Observe that this is a linear function of $X_{ij}, \overline{X_{ij}}$ with coefficients that are independent of $X_{ij}$ and controlled by Proposition~\ref{prop:tail}.

Then consider
\begin{align*}
    \hat {\vec u}^* (G(z;Q) V)^j G(z;Q) \hat {\vec v}.
\end{align*}
With the notation $a_1 = X_{ij}, a_2 = \overline{X_{ij}}$, $\vec v_1 = \vec f_{j+M}, \vec v_2 = \vec f_i$, and $\vec w_1 = \vec f_{i}, \vec w_2 = \vec f_{j+M}$ one has for $\ell = 2,3,4$
\begin{align*}
    & \sqrt{M} \hat {\vec u}^* (G(z;Q) V)^\ell G(z;Q) \hat {\vec v} \\&= \sqrt{M} \sum_{p \in \{1,2\}^\ell}\left[\left(\prod_{k=1}^\ell a_{p_k} \right) (\hat {\vec u}^* G(z;Q) \vec v_{p_1}) ( \vec w_{p_\ell}^* G(z;Q) \hat{\vec v}) \prod_{k=1}^{\ell-1} (\vec w_{p_k}^* G(z;Q) \vec v_{p_{k+1}}) \right]:= P^{(\ell)}_{ij}
\end{align*}
and set
\begin{align*}
    P^{(5)}_{ij} := \sqrt{M}\sum_{p \in \{1,2\}^5}\left[\left(\prod_{k=1}^5 a_{p_k} \right) (\hat {\vec u}^* G(z;Q) \vec v_{p_1}) ( \vec w_{p_\ell}^* G(z;X) \hat{\vec v}) \prod_{k=1}^{4} (\vec w_{p_k}^* G(z;Q) \vec v_{p_{k+1}}) \right].
\end{align*}
Whenever two vectors are orthogonal because they have disjoint support, we can replace $G(z)$ with $S(z)/\sqrt{M}$.  When $\ell$ is odd,  suppose that for a choice of $p \in \{1,2\}^\ell$ no two vectors are orthogonal in such a way.  Then $p_1 = 1$ so that $\hat {\vec u}$ is not orthogonal to $\vec v_{p_1}$.  And then $\vec w_{i}$ and $\vec v_j$ are not orthogonal if $i \neq j$, so then $p_2 = 2$, $p_3 =1$, and so on.  This implies that $p_\ell = 1$ because $\ell$ is odd.  But then $\hat {\vec v}$ is orthogonal to $\vec v_{p_\ell}$.  This implies that the order of the odd terms is actually one less than is immediately apparent. 
Write
\begin{align*}
    \hat {\vec u}^* S(z;X) \hat {\vec v} &= \hat {\vec u}^* S(z;Q) \hat {\vec v} + \sum_{k=1}^{3} M^{-k/2} J_k + M^{-5/2}J_4 =\hat {\vec u}^* S(z;Q) \hat {\vec v} + \xi, \\
    J_1 &= M^{1/2}(E_{ij} + P_{ij}^{(2)}), \quad J_2 = M P_{ij}^{(3)},\\
    J_3 &= M^{3/2} P_{ij}^{(4)}, \quad J_4 = M^{5/2} P_{ij}^{(5)}.
\end{align*}
\end{proof}

\begin{proposition}[Green's function replacement]\label{p:replace}
  Suppose $\Phi$ is an admissible test function.    Suppose further that $X$ and $Y$ are two matrices satisfying assumptions in Definition~\ref{def:SCMs} and that
  \begin{align*}
      \mathbb E X_{ij}^\ell \overline{X_{ij}}^p = \mathbb E Y_{ij}^\ell \overline{Y_{ij}}^p,
  \end{align*}
  for all choices of $\ell,p \in \mathbb N$, $\ell + p \leq 4$ and $1 \leq i \leq N, 1 \leq j \leq M$.  Then for any $\epsilon > 0$, any families of unit vectors $\{\vec q_j\}_{j=1}^n$, $\{\vec p_j\}_{j=1}^n$, and any collection of points $\{z_j\}_{j=1}^n$ bounded uniformly away from the support of the Marchenko--Pastur law $\varrho_d$ and bounded away from the real axis by $M^{-\delta}$, $1 > \delta > 0$ we have 
  \begin{align*}
      \left|\mathbb E \Phi( \hat {\vec q}_1^* S(z_1,X) \hat {\vec p}_1,\ldots, \hat {\vec q}_n^* S(z_n,X) \hat {\vec p}_n) - \mathbb E \Phi( \hat {\vec q}_1^* S(z_1,Y) \hat {\vec p}_1,\ldots, \hat {\vec q}_n^* S(z_n,Y) \hat {\vec p}_n) \right|\leq C n^5 M^{-1/2 + C' \epsilon},
\end{align*}  
where $C'> 0$ depends only $C_0$ in Definition~\ref{d:testfun}.
\end{proposition}

\newcommand{\oneto}[1]{\llbracket 1, #1 \rrbracket}
\begin{proof}
The following proof is adapted from \cite[Theorem~16.1]{Erdos2017} and \cite{Knowles2017}.
Let $\phi: \oneto{MN} \to \oneto{N} \times \oneto{M} $ be a bijection\footnote{Here $\oneto{N} = \{1,2,\ldots,N\}$.}.  For $\gamma \in \oneto{MN}$ define $X_\gamma$ by
\begin{align*}
    (X_\gamma)_{\phi(\ell)} = \begin{cases} Y_{\phi(\ell)} & \ell \leq \gamma, \\
    X_{\phi(\ell)} & \ell > \gamma. \end{cases}
\end{align*}
Note that $X_0 = X$ and $X_{MN} = Y$ and that $X_\gamma$ and $X_{\gamma+1}$ differ only in the $\phi(\gamma+1)$ entry.  Define $Q_\gamma$ by $(Q_\gamma)_{\phi(\ell)} = (X_{\gamma+1})_{\phi(\ell)}$ if $\ell \neq \gamma +1$ and $(Q_\gamma)_{\phi(\gamma + 1)} = 0$, so that $Q_\gamma$ has a zero in the exact entry where $X_\gamma$ and $X_{\gamma + 1}$ differ.  We then compare $X_\gamma$ to $Q_\gamma$ using Lemma~\ref{l:resolve} and a fifth-order Taylor expansion of $\Phi$
\begin{align*}
    &\Phi( \hat {\vec q}_1^* S(z_1,X_\gamma) \hat {\vec p}_1,\ldots, \hat {\vec q}_n^* S(z_n,X_\gamma) \hat {\vec p}_n) \\
    &= \Phi( \hat {\vec q}_1^* S(z_1,Q_\gamma) \hat {\vec p}_1 + \xi_1,\ldots, \hat {\vec q}_n^* S(z_n,Q_\gamma) \hat {\vec p}_n + \xi_n )\\
    & = \Phi( \hat {\vec q}_1^* S(z_1,Q_\gamma) \hat {\vec p}_1,\ldots, \hat {\vec q}_n^* S(z_n,Q_\gamma) \hat {\vec p}_n  )\\
    & + \sum_{k=1}^4\sum_{|\alpha| = k} \partial^\alpha \Phi( \hat {\vec q}_1^* S(z_1,Q_\gamma) \hat {\vec p}_1,\ldots, \hat {\vec q}_n^* S(z_n,Q_\gamma) \hat {\vec p}_n  ) \frac{\vec \xi^\alpha}{\alpha!} \\
    & + \sum_{|\alpha| = 5} \partial^\alpha \Phi( \hat {\vec q}_1^* S(z_1,Q_\gamma) \hat {\vec p}_1 + c\xi_1,\ldots, \hat {\vec q}_n^* S(z_n,Q_\gamma) \hat {\vec p}_n + c\xi_n ) \hat {\vec p}_n  ) \frac{\vec \xi^\alpha}{\alpha!},
\end{align*}
for some $0\leq c \leq 1$. Here $\vec \xi = (\xi_1,\ldots,\xi_n)$ and $\xi_j = \sum_{k=1}^5 M^{-k/2} J_{k,j}$, $J_{5,j} = 0$ represents the $\xi$ term in Lemma~\ref{l:resolve} applied to $\hat {\vec q_j}, \hat {\vec p_j}$, $z_j$ and $X_\gamma$.  We rewrite this expansion by collecting powers of $M^{1/2}$
\begin{align*}
    &\Phi( \hat {\vec q}_1^* S(z_1,X_\gamma) \hat {\vec p}_1,\ldots, \hat {\vec q}_n^* S(z_n,X_\gamma) \hat {\vec p}_n) \\
    & = \Phi( \hat {\vec q}_1^* S(z_1,Q_\gamma) \hat {\vec p}_1,\ldots, \hat {\vec q}_n^* S(z_n,Q_\gamma) \hat {\vec p}_n  ) + \sum_{k=1}^4 M^{-k/2} T_{k,\gamma}.
\end{align*}
By independence $\mathbb E[T_k]$ for $k \leq 4$ decomposes into a sum of terms that are a product of a quantity depending only on moments  $X_{\phi(\gamma+1)}$, $\mathbb E M^{(l+p)/2}X_{\phi(\gamma+1)}^\ell \overline{X}_{\phi(\gamma+1)}^p $, $p + \ell \leq 4$ and a quantity depending on other variables.  Then, an estimate is needed for $\Exp T_k$.

 For $\epsilon> 0$ and $D > 0$, let $\mathcal E_{Q_\gamma}$ be the event where
\begin{align*}
    \max_{k,\ell,j} [| \vec {v}_k^* G(z_\ell,Q_\gamma){\vec w}_{j}| + |{\vec v}_k^* S(z_\ell,Q_\gamma) {\vec w}_{j}|] > M^\epsilon,
\end{align*}
and the families of vectors $\{\vec v_k\}$ and $\{\vec w_k\}$ are given by the union of the families $\{\hat {\vec q}_k\}$ and $\{\hat {\vec p}_k\}$ with the standard basis vectors, respectively. 
  Then there exists a constant $C> 0$, independent of $\gamma$, such that the probability of this event is bounded above by $C M^{-D}$.  Also, let $\mathcal X_\gamma$ be the event where
  \begin{align*}
      \sqrt{M} |X_{\phi(\gamma)}| > M^\epsilon.
  \end{align*}
  We use the a priori bound $\| G(z_\ell,X_\gamma) \| \leq C M^\delta$ (see \eqref{eq:schur_norm})  and that
  \begin{align*}
      \vec {v}_k^*G(z_\ell;X_\gamma){\vec w}_{j} &= \vec {v}_k^*G(z_\ell;Q_\gamma){\vec w}_{j}\\
      &- \vec {v}_k^*G(z_\ell;Q_\gamma) (X_\gamma - Q_{\gamma}) G(z;Q_\gamma){\vec w}_{j} \\
      & +  \vec {v}_k^*G(z_\ell;Q_\gamma) (X_\gamma - Q_{\gamma})G(z_\ell;Q_\gamma) (X_\gamma - Q_{\gamma}) G(z_\ell;X_\gamma){\vec w}_{j}.
  \end{align*}
  On the event $\mathcal E_{Q_\gamma}^c \cap \mathcal X_\gamma^c$
  \begin{align*}
      |\vec {v}_k^*G(z_\ell;X_\gamma){\vec w}_{j}| \leq M^\epsilon + 2 M^{3 \epsilon -1/2} + 4 C M^{5 \epsilon - 1 + \delta}
  \end{align*}
  Using an expansion to the next order, one obtains
\begin{align*}
    |\vec {v}_k^*S(z_\ell;X_\gamma){\vec w}_{j}| \leq 2 M^{3 \epsilon } + 4 M^{5 \epsilon - 1/2} + 8C M^{7 \epsilon - 1 + \delta}
\end{align*}
  Provided that $4 \epsilon -1 + \delta \leq 0$ and  we have that
  \begin{align*}
    \max_{k,\ell,j} [| \vec {v}_k^* G(z_\ell,X_\gamma){\vec w}_{j}| + |{\vec v}_k^* S(z_\ell,X_\gamma) {\vec w}_{j}|] \leq C' M^{3\epsilon},
\end{align*}
for a new constant $C'$.

Now, consider
\begin{align*}
    \mathbb E \Phi = \mathbb E \Phi (\mathbbm 1_{\mathcal X_\gamma} + \mathbbm 1_{\mathcal X_\gamma^c} )(\mathbbm 1_{\mathcal E_{Q_\gamma}} + \mathbbm 1_{\mathcal E_{Q_\gamma}^c} ),
\end{align*}
where $|\Phi| \leq 1$, without loss of generality.   Then for every $D > 0$ there exists $C > 0$ such that
\begin{align*}
    |\mathbb E \Phi \mathbbm 1_{\mathcal X_\gamma}  \mathbbm 1_{\mathcal E_{Q_\gamma}^c} + \mathbb E \Phi \mathbbm 1_{\mathcal X_\gamma^c} \mathbbm 1_{\mathcal E_{Q_\gamma}} + \mathbb E \Phi \mathbbm 1_{\mathcal X_\gamma^c} \mathbbm 1_{\mathcal E_{Q_\gamma}^c} | \leq C M^{-D}.
\end{align*}
We need to consider
\begin{align*}
    \Exp T_{k,\gamma} \mathbbm 1_{\mathcal X_\gamma^c}\mathbbm{1}_{\mathcal E_{Q_\gamma}^c }.
\end{align*}
First,
\begin{align*}
    |\Exp T_{5,\gamma} \mathbbm{1}_{\mathcal E_{Q_\gamma}^c}\mathbbm{1}_{\mathcal E_{X_\gamma,M}^c} | \leq 1024n^5 M^{3(C_0+8)\epsilon-5/2} \max_{1 \leq k \leq 25} \mathbb E|\sqrt{M}X_{\phi(\gamma)}|^k
\end{align*}
where $M^{3C_0\epsilon}$ is the upper bound on all derivatives of $\Phi$, and $1024n^5$ is a bound on the number of terms in the Taylor expansion.  For $T_{k,\gamma}$ we note that for any $D > 0$ there exists a constant $C > 0$ such that
\begin{align*}
    |\Exp T_{k,\gamma} \mathbbm{1}_{\mathcal E_{Q_\gamma}^c}\mathbbm{1}_{\mathcal E_{X_\gamma,M}^c} -  \Exp T_{k,\gamma} \mathbbm{1}_{\mathcal E_{Q_\gamma}^c}| \leq C M^{-D}.
\end{align*}
So, we can write
\begin{align*}
    \left|\Exp \Phi( \hat {\vec q}_1^* S(z_1,X_\gamma) \hat {\vec p}_1,\ldots, \hat {\vec q}_n^* S(z_n,X_\gamma) \hat {\vec p}_n) - \sum_{k=1}^4 M^{-k/2} L_k \right| \leq C n^5 M^{3(C_0+8)\epsilon-5/2}
\end{align*}
where $L_k$ depends only on $Q_\gamma$ and the moments of $X_{\phi(\gamma)}$ up to order $4$.  The proposition follows using
\begin{align*}
    \Exp &\Phi( \hat {\vec q}_1^* S(z_1,X) \hat {\vec p}_1,\ldots, \hat {\vec q}_n^* S(z_n,X) \hat {\vec p}_n) - \Exp \Phi( \hat {\vec q}_1^* S(z_1,Y) \hat {\vec p}_1,\ldots, \hat {\vec q}_n^* S(z_n,Y) \hat {\vec p}_n)
    \\ &= \sum_{\gamma=1}^{NM} \Exp \Phi( \hat {\vec q}_1^* S(z_1,X_\gamma) \hat {\vec p}_1,\ldots, \hat {\vec q}_n^* S(z_n,X_\gamma) \hat {\vec p}_n) - \Exp \Phi( \hat {\vec q}_1^* S(z_1,X_{\gamma+1}) \hat {\vec p}_1,\ldots, \hat {\vec q}_n^* S(z_n,X_{\gamma+1}) \hat {\vec p}_n).
\end{align*}
\end{proof}

We recall well-known important facts about trapezoidal rule applied to approximate contour integrals on smooth closed curves.  Suppose $\Gamma$ is such a curve of length one with arc length parameterization $\ell : [0,1] \to \Gamma$.   We choose $\ell$ so that $\ell(0),\ell(1/2) \in \mathbb R$ and $\ell(0) < \ell(1/2)$. With $m$ points, the trapezoidal rule can be used at the nodes $t_j = j/m$ for $j = 0,1,\ldots,m$.  In our case, however, we wish to avoid evaluating on the real axis and we choose $s_j^{(m)} = s_j = (t_j + t_{j+1})/2 = (2j + 1)/(2m)$, $j = 0,1,\ldots,m$ with the convention that $s_m = s_0$.  Consider
\begin{align*}
    \oint_\Gamma f(z) \sd z &= \int_0^1 f(\ell(s)) \ell'(s) \sd s \approx \sum_{j=0}^{m-1} f(\ell(s_j)) \frac{\ell'(s_j)}{m} = \sum_{j=0}^{m-1} f(z_j) w_j,\\
    z_j^{(m)}& = z_j = \ell(s_j), \quad w_j^{(m)} = w_j = \frac{\ell'(s_j)}{m}.
\end{align*}
Using the Euler--Maclaurin formula, for every $D > 0$ there exists $C_D> 0$ such that
\begin{align*}
    \left|  \oint_\Gamma f(z) dz - \sum_{j=0}^{m-1} f(z_j) w_j \right| \leq C_D(\Gamma) \|f^{(D)}\|_\infty m^{-D}.
\end{align*}

\begin{proof}[Proof of Theorem~\ref{t:GC}]
We prove the proposition for $\Gamma_j = \Gamma$ for all $j$.  The arguments easily extend to the general case.  Let $\Phi : \mathbb C^n \to \mathbb R$ be an admissible test function.  We approximate
\begin{align*}
    \frac{\sqrt{M}}{2 \pi i} \oint_{\Gamma} f_j(z) (c_0(z;\mu_T) - s_\mfd(z)) \sd z
\end{align*}
using the trapezoidal rule and consider
\begin{align*}
    \Delta_{M,m} :=\Phi\left(\frac{\sqrt{M}}{2 \pi i}\oint_{\Gamma} f_1(z) (c_0(z; \mu_T) - s_\mfd (z)) \sd z,\ldots, \frac{\sqrt{M}}{2 \pi i}\oint_{\Gamma} f_n(z) (c_0(z; \mu_T) - s_\mfd(z))\sd z \right) \\
      -  \Phi\left( \frac{\sqrt{M}}{2 \pi i} \sum_{j=1}^m f_1(z_j) (c_0(z_j; \mu_T) - s_\mfd(z_j))w_j,\ldots, \frac{\sqrt{M}}{2 \pi i} \sum_{j=1}^m f_n(z_j) (c_0(z_j; \mu_T) - s_\mfd(z_j))w_j \right).
\end{align*}
The choice of $m$ is critical.  Examining how the conclusion of Proposition~\ref{p:replace} depends on $n$, we need $m^5 < M^{1/2}$. So, we choose $m = M^{1/20}$. 

Because $\Phi$ is bounded, for $\delta > 0$ we can restrict to the event $\mathcal L_\delta = \{ \lambda_N \geq \gamma_- - \delta, \lambda_1 \leq \gamma_+ + \delta \}$, and there exists $C_D$ such that $\mathbb P(\mathcal L_\delta)  \geq 1 - C_D M^{-D}$ for all $D > 0$.  Furthermore we choose $\delta$ so that $[\gamma_- - \delta, \gamma_+ + \delta] \subset \Omega$.    By fixing $\delta$, on this event the integrands and all their derivatives up to order $E$ are bounded by $\sqrt{M} c_E$ for some $c_E > 0$.  Then, for example, on the event $\mathcal L_\delta$
\begin{align*}
    \left | \frac{\sqrt{M}}{2 \pi i}\oint_{\Gamma} f_1(z) (c_0(z; \mu_T) - s_\mfd (z))\sd z - \frac{\sqrt{M}}{2 \pi i} \sum_{j=1}^m f_1(z_j) (c_0(z_j; \mu_T) - s_\mfd (z_j))w_j\right| \leq \frac{C_E(\Gamma) c_E \sqrt{M}}{m^E}.
\end{align*}
Since $E$ can be chosen arbitrarily large, we then find
\begin{align*}
    |\Exp \Delta_{M,m} \mathbbm 1_{\mathcal L_{\delta}}| \leq C M^{-D}
\end{align*}
for any $D > 0$.  Therefore, it suffices to consider 
\begin{align*}
   \tilde{ \Delta}_{M,m}  := \Phi\left(\sqrt{M} \sum_{j=1}^m f_1(z_j) (c_0(z_j; \mu_T) - s_\mfd(z_j))w_j,\ldots, \sqrt{M}\sum_{j=1}^m f_n(z_j) (c_0(z_j; \mu_T) - s_\mfd(z_j))w_j \right) \\
      -  \Phi\left(\sqrt{M} \sum_{j=1}^m f_1(z_j) (c_0(z_j; \mu_{\tilde T}) - s_\mfd(z_j))w_j,\ldots, \sqrt{M}\sum_{j=1}^m f_n(z_j) (c_0(z_j; \mu_{\tilde{T}}) - s_\mfd(z_j))w_j \right).
\end{align*}
And, we are led to consider the function $\Psi: \mathbb C^m \to \mathbb R$
\begin{align}\label{eq:PsiPhi}
    \Psi(x_1,x_2,\ldots,x_m) = \Phi\left( \sum_{j=1}^m f_1(z_j) \frac{w_j}{2 \pi i} x_j,\ldots,\sum_{j=1}^m f_n(z_j) \frac{w_j}{2 \pi i} x_j\right).
\end{align}
Define $W \in \mathbb C^{n \times m} $ by $W_{\ell j} = f_\ell(z_j) \frac{w_j}{2 \pi i}$ and it follows that
\begin{align*}
    \partial_{x_{j_1}x_{j_2}\cdots x_{j_q}} \Psi(x_1,\ldots,x_m) &=  \sum_{k_1,k_2,\ldots,k_q = 1}^n \partial_{y_{k_1} y_{k_2} \cdots y_{k_p}} \Phi(y_1,\ldots,y_n) \left( \prod_{p=1}^q {W_{k_p,j_p}}\right),\\
    y_k &= \sum_{j=1}^m f_k(z_j) \frac{w_j}{2 \pi i} x_j.
\end{align*}
From this, we are able to estimate
\begin{align*}
   \left| \partial_{x_{j_1}x_{j_2}\cdots x_{j_q}} \Psi(x_1,\ldots,x_m)\right| &\leq  \max_{k_1,k_2,\ldots,k_1} |\partial_{y_{k_1} y_{k_2} \cdots y_{k_p}} \Phi(y_1,\ldots,y_n)|  \sum_{k_1,k_2,\ldots,k_q = 1}^n  \prod_{p=1}^q |{W_{k_p,j_p}}|\\
   & \leq  \max_{k_1,k_2,\ldots,k_1} |\partial_{y_{k_1} y_{k_2} \cdots y_{k_p}} \Phi(y_1,\ldots,y_n)| \max_{j} \| f_j\|_\infty^q \left( \frac{C}{2 \pi}\right)^q,
\end{align*}
where $C > 0$ is such that $\sum_j |w_j| \leq C$.  Note that $C$ can be chosen independent of $m$.  Now let $\epsilon > 0$ be sufficiently small so that
\begin{align*}
    |\partial_{x_{j_1}x_{j_2}\cdots x_{j_q}} \Phi(x_1,\ldots,x_n) | \leq M^{C_0 \epsilon} ~~\text{for}~~ \max_j |x_j| \leq M^{\epsilon}
\end{align*}
All arguments for $\Phi$ in \eqref{eq:PsiPhi} are uniformly bounded by $M^{\epsilon}$  for $\max_{j} |x_j| \leq M^{\epsilon}/\left(\frac{C}{2 \pi} \max_j \|f_j\|_\infty\right)$.  Thus
\begin{align*}
    |\partial_{y_{k_1} y_{k_2} \cdots y_{k_p}} \Psi(x_1,\ldots,x_\alpha)|  \leq M^{C_0\epsilon} \max_{j} \| f_j\|_\infty^p \left( \frac{C}{2 \pi}\right)^p
\end{align*}
By setting $L = \frac{C}{2 \pi} \max_j \|f_j\|_\infty$ we find that
\begin{align*}
    \tilde \Psi(x_1,x_2,\ldots,x_m) = \Phi\left( L^{-1} \sum_{j=1}^m f_1(z_j) \frac{w_j}{2 \pi i} x_j,\ldots,L^{-1}\sum_{j=1}^m f_n(z_j) \frac{w_j}{2 \pi i} x_j\right).
\end{align*}
is admissible with the same constant $C_0$.  Applying Proposition~\ref{p:replace} to $\tilde \Psi$ establishes the proposition.
\end{proof}

We also remark that these arguments, without the use of Proposition~\ref{p:replace}, can be used to show the following:

\begin{proposition}\label{prop:univ-lo}
  Suppose $W$ is a sample covariance matrix, $N/M \To[M] d \in (0,1) $ and $T = T(W,\vec b)$ for a sequence $\vec b = \vec b_N \in \mathbb C^N$ of non-trivial vectors.  Then
\begin{align*}
    \left(\int \lambda^k \mu_T(\sd \lambda)\right)_k \Wkto[M] \left( \int \lambda^k  \varrho_\mfd(\sd \lambda) \right)_k, 
\end{align*}
in the sense of convergence of finite-dimensional marginals where $k \geq 0$ if $d = 1$ and $k \in \mathbb Z$ if $d < 1$.
\end{proposition}

\section{Analysis of the algorithms}\label{sec:analysis}

The important fact that we use to prove Theorems~\ref{t:main-lo} and \ref{t:main} is that the entries in the Cholesky factorization of the three-term recurrence matrix associated to a measure $\mu$ are (generically) differentiable functions of the moments of the measure.  This implies that the leading-order behavior (Theorem~\ref{t:main-lo}) is the same as in the Gaussian case and that, with the moment matching condition (Definition~\ref{def:moment}), the fluctuations must be the same as in the Gaussian case (Theorem~\ref{t:main}).  So, it suffices to prove Theorem~\ref{t:main} in the case of $X$ having $\mathcal N_{\beta}(0,1/M)$ entries.  The following three sections do just this.

\subsection{Proofs for the conjugate gradient algorithm}

The basis for our analysis is Proposition~\ref{prop:errors} and Theorem~\ref{t:lanczos}.   In this section we suppose $W \lawequals \mathcal W_{\beta}(N,M)$, $N \leq M$ and $\vec b = \vec b_N \in \mathbb C^{N}$ (or $\mathbb R^N$ if $\beta = 1$).  And we recall the notation that $\vec x_k = \vec x_k(W,\vec b)$ is the $k$-th iterate of the CGA applied to $W \vec x = \vec b$ and $\vec r_k = \vec b - W \vec x_k$, $\vec e_k = \vec x - \vec x_k$.  

\subsubsection{Non-asymptotic calculations}\label{sec:CG-non}

Using the notation \eqref{eq:Hdef}, with $T = T(W,\vec b) = HH^T$ it follows that
\begin{align*}
    \pi_k(0;\mu_T) = (-1)^{k+1}\prod_{j=0}^{k-1} \alpha_j^2, \quad T_{j,j+1} = \alpha_j \beta_j,
\end{align*}
and therefore
\begin{align}\label{eq:rk_form_CG}
    \|\vec r_k\|_2^2 = \prod_{j=0}^{k-1} \frac{\beta_j^2}{\alpha_j^2},
\end{align}
where the chi squared random variables are all mutually independent.  This formula lends itself easily to asymptotic analysis.

Deriving a distributional expression for $\|\vec e_k\|_W^2$  is more involved.  With the convention that $b_{-1} = 1$
\begin{align*}
    \begin{bmatrix} \pi_{k+1}(x;\mu_T) \\ \pi_{k}(0;\mu_T) \end{bmatrix} = \begin{bmatrix}x-a_k & -b_{k-1}^2 \\ 1 & 0 \end{bmatrix}\begin{bmatrix}x-a_{k-1} & -b_{k-2}^2 \\ 1 & 0 \end{bmatrix} \cdots \begin{bmatrix}x-a_0 & -b_{-1}^2 \\ 1 & 0 \end{bmatrix} \begin{bmatrix} 1 \\ 0 \end{bmatrix},
    \end{align*}
    Then define the complementary polynomials
    \begin{align*}
        \begin{bmatrix} \tilde \pi_{k+1}(x;\mu_T) \\ \tilde \pi_{k}(x;\mu_T) \end{bmatrix} &= \begin{bmatrix}x-a_k & -b_{k-1}^2 \\ 1 & 0 \end{bmatrix}\begin{bmatrix}x-a_{k-1} & -b_{k-2}^2 \\ 1 & 0 \end{bmatrix} \cdots \begin{bmatrix}x-a_0 & -b_{-1}^2 \\ 1 & 0 \end{bmatrix} \begin{bmatrix} 0 \\ 1 \end{bmatrix}.
    \end{align*}
Decompose
\begin{align*}
c_k(0;\mu_T) = c_0(0;\mu_T) \pi_k(0;\mu_T) - \tilde \pi_k(0;\mu_T).
\end{align*}
Then
\begin{align*}
\tilde \pi_k(0;\mu_T) = (-1)^{k+1}\sum_{\ell=0}^{k-1} \left(\prod_{j=1}^\ell \beta_{j-1}^2\right)\left( \prod_{j=\ell+1}^{k-1} \alpha_j^2\right),
\end{align*}
giving
\begin{align*}
    \frac{\tilde \pi_k(0;\mu_T)}{\pi_k(0;\mu_T)} = \frac{1}{\alpha_0^2}\sum_{\ell=0}^{k-1} \prod_{j=1}^{\ell} \frac{\beta_{j-1}^2}{\alpha_j^2},
\end{align*}
where the empty product returns one. From Proposition~\ref{p:finite}
\begin{align*}
    c_0(0,\mu_T) = \frac{1}{\alpha_0^2} \left( 1 +  \sum_{\ell=1}^{k-1} \prod_{j = 1}^\ell \frac{\beta_{j-1}^2}{\alpha_{j}^2} + \left(\beta_{k-1}^2  \Sigma_k^{-2} \right)\prod_{j = 1}^{k-1} \frac{\beta_{j-1}^2}{\alpha_{j}^2}\right),
\end{align*}
where $\Sigma_k \lawequals \frac{\chi_{\beta(M -N +1)}}{\sqrt{\beta M}}$ is independent of $(\alpha_j,\beta_j)_{j = 0}^{k-1}$. We find
\begin{align}\label{eq:ek_form}
    \|\vec e_k\|_W^2 = \Sigma_k^{-2} \prod_{j=0}^{k-1} \frac{\beta_j^2}{\alpha_j^2}.
\end{align}
where the chi squared random variables are all mutually independent.  This establishes Theorem~\ref{t:GaussianOnly}(a) and Theorem~\ref{t:deterministic} follows as well.

\subsubsection{Asymptotic calculations}

\begin{proof}[Proof of Theorems~\ref{t:main-lo}(a) and \ref{t:main}(a) when $\sqrt{M} X \lawequals \mathcal G_{\beta}(N,M)$]

Decompose
\begin{align*}
c_0(0;\mu_T) = \frac{1}{1-\mfd} + \frac{\sqrt{2}}{\sqrt{\beta M}} R(\mu_T), \quad T = T(W,\vec b),
\end{align*}
using Proposition~\ref{prop:dist-limit}.  In the notation of this proposition $R(\mu_T) \Wkto[M] Z_{-1}/\sqrt{2}$.  Then using the complementary polynomials
\begin{align*}
    \frac{c_k(0;\mu_{T})}{\pi_k(0;\mu_{T})} =  \frac{(1-\mfd)^{-1} \pi_k(0;\mu_T) - \tilde \pi_k(0;\mu_T)}{\pi_k(0;\mu_T)} + \frac{\sqrt{2}}{\sqrt{\beta M}} R(\mu_T).
\end{align*}

We write $T = HH^T$, again using the notation \eqref{eq:Hdef}.  Using the distributional limit described in Proposition~\ref{prop:dist-limit} one can compute the large $N$ behavior.  Specifically, we use \eqref{eq:ajs} and \eqref{eq:bjs} extensively.  Using the same process $(Z_j)_{j\geq 1}$ write
\begin{align*}
    A_k:=\begin{bmatrix} - a_k & - b_{k-1}^2 \\ 1 & 0 \end{bmatrix} &= \hat E + \sqrt{ \frac{2}{\beta M} } \check E_k + o(M^{-1/2}),\\
    \hat E &= \begin{bmatrix} -1 -\mfd & -\mfd \\ 1 & 0 \end{bmatrix},\\
    \check E_k &= \begin{bmatrix} - Z_{2k+1} - \sqrt{\mfd} Z_{2k} & -\mfd Z_{2k-1} - \sqrt{\mfd} Z_{2k} \\
    0 & 1 \end{bmatrix}
\end{align*}
We compute the asymptotics of the quantity
\begin{align*}
    \begin{bmatrix} 1 & 0 \end{bmatrix} A_{k-1} A_{k-2} \cdots A_1 \begin{bmatrix} - \alpha_0^2 & -1 \\ 1 & 0 \end{bmatrix} \begin{bmatrix} \frac{1}{1-\mfd} \\ -1 \end{bmatrix}
\end{align*}
using that
\begin{align*}
    \hat E &= V \Lambda V^{-1}, \quad V = \begin{bmatrix} -1 & -\mfd \\ 1 & 1 \end{bmatrix}, \quad V^{-1} = \frac{1}{1-\mfd} \begin{bmatrix} -1 & -\mfd \\ 1 & 1 \end{bmatrix}, \\
    \Lambda &= \mathrm{diag}(-1,-\mfd).
\end{align*}
We find that 
\begin{align*}
    &\begin{bmatrix} 1 & 0 \end{bmatrix} \hat E^{k-j-1} \check E_j \hat E^{j-1} \begin{bmatrix} - 1 & -1 \\ 1 & 0 \end{bmatrix} \begin{bmatrix} \frac{1}{1-\mfd} \\ -1 \end{bmatrix} \\
    &= (-1)^{k+1}  \frac{\mfd^j-\mfd^{k}}{(1-\mfd)^2} \left[Z_{2j+1} + \sqrt{\mfd}\, Z_{2j} - Z_{2j-1} - Z_{2j}/\sqrt{\mfd} \right].
\end{align*}
Similarly,
\begin{align*}
\begin{bmatrix} 1 & 0 \end{bmatrix} \hat E^{k-1} \begin{bmatrix} - Z_1 & -1 \\ 1 & 0 \end{bmatrix} \begin{bmatrix} \frac{1}{1-\mfd} \\ -1 \end{bmatrix} = (-1)^{k} \frac{ 1-\mfd^k}{(1-\mfd)^2}Z_1.
\end{align*}
 
Therefore it remains to analyze
\begin{align*}
    \prod_{j=0}^{k-1} \alpha_j^2 &= 1 + \frac{\sqrt{2}}{\sqrt{\beta M}} \sum_{j=0}^{k-1} Z_{2j+1} + o(M^{-1/2}),\\
    \prod_{j=0}^{k-1} \beta_j^2 &= \mfd^k + \frac{\sqrt{2}}{\sqrt{\beta M}} \sum_{j=0}^{k-1} \mfd^{k-1/2} Z_{2j+2} + o(M^{-1/2}).
\end{align*}
The distributional limit of $R(\mu_T)$ is provided by Proposition~\ref{prop:dist-limit}.  So, our final expressions become
\begin{align*}
    \|\vec e_k(W,\vec b)\|^2_W =& \frac{\mfd^k}{1-\mfd} \left( 1 + \frac{\sqrt{2}}{\sqrt{\beta M}} \left[ \sum_{j=k}^\infty \mfd^{j-k} (Z_{2j}/\sqrt{\mfd} - Z_{2j+1}) \right.\right. \\
    & \left.\left. + \sum_{j=1}^{k-1} (Z_{2j}/\sqrt{\mfd} - Z_{2j-1}) - Z_{2k-1}\right] \right) + o(M^{-1/2}),\\
    \|\vec r_k(W,\vec b)\|^2_2 =& {\mfd^k} \left( 1 + \frac{\sqrt{2}}{\sqrt{\beta M}}\left[\sum_{j=0}^{k-1} \left(Z_{2j+2}/\sqrt{\mfd} -  Z_{2j+1} \right)\right]  \right) + o(M^{-1/2}).
\end{align*}
The theorem follows.
\end{proof}

\subsection{Proofs for the MINRES algorithm}

\subsubsection{Non-asymptotic calculations}\label{sec:MINRES-non}

The proof of Theorem~\ref{t:deterministic}(b) is immediate from the simple formula
\begin{align*}
    p_j(0,\mu_T) = \frac{\det( -T_j)}{\prod_{\ell = 0}^{j-1} b_j^2} =  (-1)^j \prod_{\ell = 0}^{j-1} \frac{\beta_\ell}{\alpha_\ell},
\end{align*}
using \eqref{eq:Hdef}.  Then using \eqref{eq:dist_H}, Theorem~\ref{t:GaussianOnly}(b) follows.

\subsubsection{Asymptotic calculations}

\begin{proof}[Proof of Theorems~\ref{t:main-lo}(b) and \ref{t:main}(b) when $\sqrt{M} X \lawequals \mathcal G_{\beta}(N,M)$]
It suffices to prove Theorem~\ref{t:main}(b) in this case.  From Theorem~\ref{t:main}(b) we have that 
\begin{align*}
    \prod_{\ell = 0}^{j-1} \frac{\beta^2_\ell}{\alpha^2_\ell} = \mfd^j\left(1   + \frac{\sqrt{2}}{\sqrt{\beta M}} \tilde Z_j^{\vec r, \mathrm{CG}}\right) + o(M^{-1/2}), \quad \tilde Z_j^{\vec r, \mathrm{CG}} = d^{-j} Z_j^{\vec r, \mathrm{CG}}.
\end{align*}
This implies
\begin{align*}
    \sum_{j=0}^{k}\prod_{\ell = 0}^{j-1} \frac{\alpha^2_\ell}{\beta^2_\ell} = \sum_{j=0}^{k} \mfd^{-j} - \frac{\sqrt{2}}{\sqrt{\beta M}}\sum_{j=1}^{k}\mfd^{-j}\tilde Z_j^{\vec r, \mathrm{CG}} + o(M^{-1/2}).
\end{align*}
And this gives
\begin{align*}
\left(\sum_{j=0}^{k}\prod_{\ell = 0}^{j-1} \frac{\alpha^2_\ell}{\beta^2_\ell}\right)^{-1} =& \frac{1 - \mfd^{-1}}{1-\mfd^{-k-1}} + \left(\frac{1 - \mfd^{-1}}{1-\mfd^{-k-1}}\right)^2 \frac{\sqrt{2}}{\sqrt{\beta M}}\sum_{j=1}^{k}\mfd^{-j}\tilde Z_j^{\vec r, \mathrm{CG}}\\
&+ o(M^{-1/2}).
\end{align*}
In writing
\begin{align*}
    \frac{1 - \mfd^{-1}}{1-\mfd^{-k-1}} = \mfd^k\frac{1 -\mfd}{1-\mfd^{k+1}},
\end{align*}
we establish the theorem.
\end{proof}

\subsection{Proofs for the conjugate gradient algorithm applied to the normal equations}\label{sec:CG-ne}

First, observe that for $\alpha > 0$
\begin{align}\label{eq:op_obs}
    p_j(\lambda;\alpha \mu) = \frac{p_j(\lambda; \mu)}{\sqrt{\alpha}}.
\end{align}
Consider the distribution of the measure $\nu$ as defined in \eqref{eq:nu}, and in particular, the distribution on the absolute value of the vector $V^* \vec b$  where $X = U \Sigma V^*$ is  the singular value decomposition of $X$.  We know that $V$ can be taken to be Haar distributed on either the orthogonal ($\beta = 1$) or unitary ($\beta = 2$) group \cite{Edelman2005}.  By invariance, if $\|\vec b \|_2 = 1$ then $\vec b$ can be replaced with $\vec f_1$.   From this it follows that for $T = T(W,\vec a)$
\begin{align*}
    \mu_T \lawequals \sum_{j=1}^N \omega_j \delta_{\lambda_j}, \quad \omega_j \lawequals \frac{ \chi_{\beta,j}^2 }{\displaystyle\sum_{\ell=1}^N \chi_{\beta,\ell}^2 }, \quad j = 1,2,\ldots,N,
\end{align*}
and $(\lambda_1,\ldots,\lambda_N)$ are the eigenvalues of $W$ which are independent of $(\omega_1,\ldots,\omega_N)$.  So, we find that, in the notation of Theorem~\ref{t:GaussianOnly}(c)
\begin{align*}
    \nu \lawequals \underbrace{\left(\frac{\displaystyle\sum_{\ell=1}^N \chi_{\beta,\ell}^2 }{\displaystyle\sum_{\ell=1}^M \chi_{\beta,\ell}^2 }\right)}_{\Delta_{N,M}} \sum_{j=1}^N \omega_j \delta_{\lambda_j}.
\end{align*}
Combined with \eqref{eq:op_obs}, this gives the proof of Theorem~\ref{t:GaussianOnly}(c).  And then Theorem~\ref{t:main}(c) and Theorem~\ref{t:main-lo}(c), in the Gaussian case, follow.

\bibliographystyle{amsalpha}
\bibliography{references,ERefs}

\providecommand{\bysame}{\leavevmode\hbox to3em{\hrulefill}\thinspace}
\providecommand{\MR}{\relax\ifhmode\unskip\space\fi MR }
\providecommand{\MRhref}[2]{%
  \href{http://www.ams.org/mathscinet-getitem?mr=#1}{#2}
}
\providecommand{\href}[2]{#2}
\begin{thebibliography}{DMOT14b}

\bibitem[BK01]{Beckermann2001}
B~Beckermann and A~B~J Kuijlaars, \emph{{Superlinear Convergence of Conjugate
  Gradients}}, SIAM Journal on Numerical Analysis \textbf{39} (2001), no.~1,
  300--329.

\bibitem[Bor87]{Borgwardt1987}
K~H Borgwardt, \emph{{The simplex method: A probabilistic analysis}},
  Springer--Verlag, Berlin, Heidelberg, 1987.

\bibitem[BS04]{BaiSilverstein2004}
Z~D Bai and J~W Silverstein, \emph{C{LT} for linear spectral statistics of
  large-dimensional sample covariance matrices}, Ann. Probab. \textbf{32}
  (2004), no.~1A, 553--605. \MR{2040792}

\bibitem[BS10]{Bai2010}
Z~Bai and J~W Silverstein, \emph{{Spectral Analysis of Large Dimensional Random
  Matrices}}, Springer Series in Statistics, Springer New York, New York, NY,
  2010.

\bibitem[DE02]{Dumitriu2002}
I~Dumitriu and A~Edelman, \emph{{Matrix models for beta ensembles}}, Journal of
  Mathematical Physics \textbf{43} (2002), no.~11, 5830.

\bibitem[DE06]{DumitriuEdelman2006}
\bysame, \emph{Global spectrum fluctuations for the {$\beta$}-{H}ermite and
  {$\beta$}-{L}aguerre ensembles via matrix models}, J. Math. Phys. \textbf{47}
  (2006), no.~6, 063302, 36. \MR{2239975}

\bibitem[Dei00]{DeiftOrthogonalPolynomials}
P~Deift, \emph{{Orthogonal Polynomials and Random Matrices: a Riemann-Hilbert
  Approach}}, Amer. Math. Soc., Providence, RI, 2000.

\bibitem[DMOT14a]{Deift2014a}
P~A Deift, G~Menon, S~Olver, and T~Trogdon, \emph{{Universality in numerical
  computations with random data}}, Proceedings of the National Academy of
  Sciences \textbf{111} (2014), no.~42, 14973--14978.

\bibitem[DMOT14b]{Deift2014}
\bysame, \emph{{Universality in numerical computations with random data}},
  Proceedings of the National Academy of Sciences of the United States of
  America \textbf{111} (2014), no.~42, 14973--8.

\bibitem[DMT16]{Deift2015}
P~A Deift, G~Menon, and T~Trogdon, \emph{{On the condition number of the
  critically-scaled Laguerre Unitary Ensemble}}, Discrete and Continuous
  Dynamical Systems \textbf{36} (2016), no.~8, 4287--4347.

\bibitem[DS01]{Davidson2001}
KR~Davidson and S~J Szarek, \emph{{Local Operator Theory, Random Matrices and
  Banach Spaces}}, Handbook of the Geometry of Banach Spaces, Elsevier, 2001,
  pp.~317--366.

\bibitem[DS15]{DuyShirai}
T~K Duy and T~Shirai, \emph{The mean spectral measures of random {J}acobi
  matrices related to {G}aussian beta ensembles}, Electron. Commun. Probab.
  \textbf{20} (2015), no. 68, 13. \MR{3407212}

\bibitem[DT17]{Deift2017b}
P~Deift and T~Trogdon, \emph{{Universality for Eigenvalue Algorithms on Sample
  Covariance Matrices}}, SIAM Journal on Numerical Analysis \textbf{55} (2017),
  no.~6, 2835--2862.

\bibitem[DT18a]{Deift2018}
\bysame, \emph{{Universality for the Toda Algorithm to Compute the Largest
  Eigenvalue of a Random Matrix}}, Communications on Pure and Applied
  Mathematics \textbf{71} (2018), no.~3, 505--536.

\bibitem[DT18b]{Deift2017c}
\bysame, \emph{{Universality in numerical computation with random data: Case
  Studies, Analytical Results and Some Speculations}}, Abel Symposia, vol.~13,
  3 2018, pp.~221--231.

\bibitem[DT19]{Deift2019b}
\bysame, \emph{{The conjugate gradient algorithm on well-conditioned Wishart
  matrices is almost deteriministic}}, arXiv preprint arXiv:1901.09007 (2019).

\bibitem[Duy18]{Duy}
T~K Duy, \emph{On spectral measures of random {J}acobi matrices}, Osaka J.
  Math. \textbf{55} (2018), no.~4, 595--617. \MR{3862777}

\bibitem[ER05]{Edelman2005}
A~Edelman and N~R Rao, \emph{{Random matrix theory}}, Acta Numerica \textbf{14}
  (2005), 233--297.

\bibitem[EY17]{Erdos2017}
L~Erd{\H{o}}s and H-T Yau, \emph{{Dynamical approach to random matrix theory}},
  Amer. Math. Soc., Providence, RI, 2017.

\bibitem[Gem80]{Geman1980}
Stuart Geman, \emph{{A limit theorem for the norm of random matrices}}, The
  Annals of Probability \textbf{8} (1980), no.~2, 252--261 (EN).

\bibitem[GvN51]{Goldstine1951}
H~H Goldstine and J~von Neumann, \emph{{Numerical inverting of matrices of high
  order. II}}, Proceedings of the AMS \textbf{2} (1951), no.~2, 188--202 (EN).

\bibitem[HS52]{Hestenes1952}
M~Hestenes and E~Steifel, \emph{{Method of Conjugate Gradients for Solving
  Linear Systems}}, J. Research Nat. Bur. Standards \textbf{20} (1952),
  409--436.

\bibitem[Joh98]{Johansson1998}
Kurt Johansson, \emph{{On fluctuations of eigenvalues of random Hermitian
  matrices}}, Duke Mathematical Journal \textbf{91} (1998), no.~1, 151--204.

\bibitem[Kui06]{Kuijlaars2006}
A~B~J Kuijlaars, \emph{{Convergence Analysis of Krylov Subspace Iterations with
  Methods from Potential Theory}}, SIAM Review \textbf{48} (2006), no.~1,
  3--40.

\bibitem[KY17]{Knowles2017}
A~Knowles and J~Yin, \emph{{Anisotropic local laws for random matrices}},
  Probability Theory and Related Fields \textbf{169} (2017), no.~1-2, 257--352.

\bibitem[Meu19]{Meurant2019}
G~Meurant, \emph{{On prescribing the convergence behavior of the conjugate
  gradient algorithm}}, Numerical Algorithms (2019).

\bibitem[MT16]{Menon2016}
G~Menon and T~Trogdon, \emph{{Smoothed analysis for the conjugate gradient
  algorithm}}, SIGMA \textbf{12} (2016), 1--19.

\bibitem[Mui82]{Muirhead1982}
R~J Muirhead, \emph{{Aspects of Multivariate Statistical Theory}}, Wiley Series
  in Probability and Statistics, John Wiley {\&} Sons, Inc., Hoboken, NJ, USA,
  1982.

\bibitem[OLBC10]{DLMF}
F~W~J Olver, D~W Lozier, R~F Boisvert, and C~W Clark, \emph{{NIST Handbook of
  Mathematical Functions}}, Cambridge University Press, 2010.

\bibitem[ORS13]{ORourkeRenfrewSoshnikov2013}
S~O'Rourke, D~Renfrew, and A~Soshnikov, \emph{On fluctuations of matrix entries
  of regular functions of {W}igner matrices with non-identically distributed
  entries}, J. Theoret. Probab. \textbf{26} (2013), no.~3, 750--780.
  \MR{3090549}

\bibitem[ORS14]{ORourkeRenfrewSoshnikov2014}
\bysame, \emph{Fluctuations of matrix entries of regular functions of sample
  covariance random matrices}, Theory Probab. Appl. \textbf{58} (2014), no.~4,
  615--639. \MR{3403019}

\bibitem[PDM14]{DiagonalRMT}
C~W Pfrang, P~Deift, and G~Menon, \emph{{How long does it take to compute the
  eigenvalues of a random symmetric matrix?}}, Random matrix theory,
  interacting particle systems, and integrable systems, MSRI Publications
  \textbf{65} (2014), 411--442.

\bibitem[PvMP20]{Paquette2020a}
C~Paquette, B~van Merri{\"{e}}nboer, and F~Pedregosa, \emph{{Halting Time is
  Predictable for Large Models: A Universality Property and Average-case
  Analysis}}.

\bibitem[Shc11]{Shcherbina2011}
M~Shcherbina, \emph{Central limit theorem for linear eigenvalue statistics of
  the {W}igner and sample covariance random matrices}, Zh. Mat. Fiz. Anal.
  Geom. \textbf{7} (2011), no.~2, 176--192, 197, 199. \MR{2829615}

\bibitem[Sil85]{Silverstein1985}
J~W Silverstein, \emph{{The Smallest Eigenvalue of a Large Dimensional Wishart
  Matrix}}, The Annals of Probability \textbf{13} (1985), no.~4, 1364--1368.

\bibitem[Sma83]{Smale1983}
S~Smale, \emph{{On the average number of steps of the simplex method of linear
  programming}}, Mathematical Programming \textbf{27} (1983), no.~3, 241--262.

\bibitem[SST06]{Sankar2006}
A~Sankar, D~A Spielman, and S-H Teng, \emph{{Smoothed Analysis of the Condition
  Numbers and Growth Factors of Matrices}}, SIAM Journal on Matrix Analysis and
  Applications \textbf{28} (2006), no.~2, 446--476 (en).

\bibitem[ST01]{Spielman2001}
D~Spielman and S-H Teng, \emph{{Smoothed analysis of algorithms}}, Proceedings
  of the thirty-third annual ACM symposium on Theory of computing - STOC '01
  (New York, New York, USA), ACM Press, 2001, pp.~296--305.

\bibitem[TBI97]{TrefethenBau}
L~N Trefethen and D~Bau~III, \emph{{Numerical linear algebra}}, Society for
  Industrial and Applied Mathematics (SIAM), Philadelphia, PA, 1997.

\bibitem[Ver09]{Vershynin2009}
R~Vershynin, \emph{{Introduction to the non-asymptotic analysis of random
  matrices}}, Compressed Sensing (Yonina~C. Eldar and Gitta Kutyniok, eds.),
  Cambridge University Press, Cambridge, 2009, pp.~210--268.

\bibitem[VK19]{Vargas2019}
J~G Vargas and A~Kulkarni, \emph{{The Lanczos Algorithm Under Few Iterations:
  Concentration and Location of the Ritz Values}}, arXiv preprint
  arXiv:1904.06012 (2019).

\bibitem[Wis28]{WISHART1928}
J~Wishart, \emph{{The generalised product moment distribution in samples from a
  normal multivariate population}}, Biometrika \textbf{20A} (1928), no.~1-2,
  32--52.

\end{thebibliography}

\end{document}